\documentclass[12pt]{amsart}

\usepackage[utf8]{inputenc}
\usepackage[margin=1in]{geometry}
\usepackage{amsmath,amsthm,amssymb}
\usepackage{enumitem}

\usepackage{tikz,float}
\usepackage{tikz-3dplot}
\usepackage{tikz-cd}
\usetikzlibrary{calc,positioning,shapes.geometric,arrows,decorations.markings}
\usetikzlibrary{fit}

\usepackage{bbm}
\usepackage[colorlinks]{hyperref}
\usepackage{xurl}
\hypersetup{breaklinks=true}
\usepackage{cleveref}

\numberwithin{equation}{section}
\numberwithin{figure}{section}

\newtheorem{theorem}{Theorem}[section]
\newtheorem{prop}[theorem]{Proposition}

\newtheorem{corollary}[theorem]{Corollary}

\theoremstyle{definition}
\newtheorem{remark}[theorem]{Remark}
\newtheorem{definition}[theorem]{Definition}

\newtheorem{example}[theorem]{Example}

\newcommand{\Def}[1]{\emph{#1}}

\newcommand*{\GP}{\mathbf{GP}} 
\newcommand*{\nGP}{\overline{\GP}}
\renewcommand*{\P}{\mathbf{P}} 
\newcommand*{\A}{\mathbf{A}} 
\newcommand*{\nA}{\overline{\A}}
\newcommand*{\Or}{\mathcal{O}}
\newcommand*{\OP}{\mathbf{OP}}
\newcommand*{\nOP}{\overline{\OP}}
\newcommand*{\nPI}{\overline{\mathbf{\Pi}}}
\newcommand{\R}{\mathbb{R}}
\newcommand*\1{\mathbbm{1}}
\newcommand{\poset}{\mathrm{poset}}
\newcommand{\conv}[1]{\mathrm{conv}(#1)}
\newcommand{\cone}{\mathrm{cone}}

\newcommand*{\id}{\mathrm{id}}
\newcommand{\pp}[1]{\mathcal{P}( #1 )}
\newcommand{\mv}[1]{\mathrm{max}( #1 )}
\newcommand{\sv}[1]{\mathrm{sym}( #1 )}

\definecolor{lgreen} {RGB}{180,210,100}
\definecolor{dblue}  {RGB}{20,66,129}
\definecolor{ddblue} {RGB}{11,36,69}
\definecolor{lred}   {RGB}{220,0,0}
\definecolor{nred}   {RGB}{224,0,0}
\definecolor{norange}{RGB}{230,120,20}
\definecolor{nyellow}{RGB}{255,221,0}
\definecolor{ngreen} {RGB}{98,158,31}
\definecolor{dgreen} {RGB}{78,138,21}
\definecolor{nblue}  {RGB}{28,130,185}
\definecolor{jblue}  {RGB}{20,50,100}
\definecolor{lgray}  {RGB}{211,211,211}
\definecolor{nnorange}{RGB}{245,193,141}
\hypersetup{
    colorlinks,
    citecolor=norange,
    linkcolor=norange
}
\definecolor{npurple}{HTML}{702963}
\definecolor{bblue}{HTML}{0671B7}
\definecolor{npink}{HTML}{e64072}

\graphicspath{ {./figures/} }
\title{Algebra and combinatorics of the Brion map on generalized permutahedra}
\author{Alvaro Cornejo, Mariel Supina}

\begin{document}

\begin{abstract}
The Brion morphism maps a generalized permutahedron to a collection of posets associated to its vertices.
We compute this map explicitly for the Hopf monoids of permutahedra, associahedra, and orbit polytopes, and we explore the dual Brion map of the primitive Lie monoids associated to these three Hopf monoids.
We describe the Lie monoid structure of the primitives in this dual setting and in particular we show that the Lie monoid of primitives of associahedra is isomorphic to the positive part of the Witt Lie algebra.
\end{abstract}

\maketitle

\section{Introduction}

Joyal \cite{JOYAL1981}, Joni and Rota \cite{joni_rota_1982}, Schmitt \cite{schmitt_1993}, and others used Hopf algebras as an algebraic framework to study combinatorial objects, like posets, graphs, and matroids, as these objects have natural operations of merging and breaking.
Aguiar and Mahajan \cite{AguiarMahajan} provided a useful framework to study combinatorial objects in the contexts of species and Hopf monoids, which keep track of more combinatorial information than Hopf algebras. 
Aguiar and Ardila \cite{AguiarArdila} showed that the family of polytopes called generalized permutahedra have the structure of a Hopf monoid and is the largest family of polytopes that support this structure.
They also showed that generalized permutahedra contain many other combinatorial Hopf monoids involving posets, graphs, matroids, and more. 

When studying algebraic structures, it is often useful to study its structure in two ways: internally by studying elements of the structure themselves, and externally by seeing how structure-preserving maps behave.
We explore these two approaches on $\nGP$, the Hopf monoid of normally equivalent generalized permutahedra, and the submonoids $\nPI$, $\nA$, and $\nOP$ involving permutahedra, associahedra, and orbit polytopes respectively, and a Hopf monoid morphism called Brion map from $\nGP$ to the Hopf monoid of posets $\P$.

In the internal perspective, in a cocommutative Hopf monoid $\mathbf{H}$, the primitive elements $\mathcal{P}(\mathbf{H})$ play an important role due to a variant of the Cartier--Milnor--Moore theorem \cite{Cartier2007,MilnorMoore,AguiarMahajan2}.
In the context of Hopf monoids this theorem states that if we have a cocommutative Hopf monoid, then $\mathbf{H}$ is isomorphic to the universal enveloping monoid on the primitives. 
While $\nGP$ is not cocommutative, it is commutative. 
This motivates us to consider its dual, which is cocommutative, allowing us to recover the dual Hopf monoid $\nGP^*$ from $\mathcal{P}(\nGP^*)$. 
There is also additional structure over a Hopf monoid and its primitive elements. 
A Hopf monoid $\mathbf{H}$ is a Lie monoid with Lie bracket given by the commutator $x \cdot y - y \cdot x$, and the primitives $\pp{\mathbf{H}}$ form a Lie submonoid.

We give the necessary background of Hopf monoids, duality,  and the Cartier--Milnor--Moore theorem in Section~\ref{sec:Background}. In Section~\ref{sec:hopfmonoidexamples}, we explain the Hopf monoid of posets $\P$, the Hopf monoid of generalized permtuahedra $\nGP$ along with the necessary polyhedral background, and the Hopf monoids $\nPI$, $\nA$, and $\nOP$.
In Section~\ref{sec:lieprimitives}, we apply the general theory to describe the coproduct and primitives of the dual Hopf monoids mentioned previously.
In Theorem~\ref{thm:posetprimitives} we classify the primitives of $\P^*$.
Moreover, in Theorem~\ref{thm:GPprimitives} we describe the primitives of $\nGP^*$ which consequently gives primitives of the submonoids $\nPI^*$, $\nA^*$, and $\nOP^*$ (Corollary~\ref{cor:subGPprimitives}). 
We also compute the Lie brackets of the primitives of $\nPI^*$ (Theorem~\ref{thm:liebracketpermutahedra}), of $\nA^*$ (Theorem~\ref{thm:liebracketassociahedra}), and of $\nOP^*$ (Theorem~\ref{thm:liebracketorbit}). 
In particular, for the Hopf monoid involving associahedra we find that the associated Lie algebra is isomorphic to the positive part of the Witt Lie algebra (Theorem~\ref{thm:liealgebraassociahedra}). 

In the external perspective, Ardila and Sanchez \cite{ArdilaSanchez} introduced several morphisms of Hopf monoids such as the Brianchon--Gram morphism and the Brion morphism.
The latter was motivated by Brion's theorem, which expresses the generating function of the lattice points of a polytope in terms of the generating function for the tangent cones of the vertices. In the context of Hopf monoids, the Brion map is a Hopf monoid morphism from $\nGP$ to $\P$, as the tangent cones of the vertices of generalized permutahedra can be encoded by posets. In Section~\ref{sec:brionmap}, we calculate the Brion map when restricting to the submonoids $\nPI$ (Theorem~\ref{thm:brionmappermutahedra}), $\nA$ (Theorem~\ref{thm:brionmapassociahedra}), and $\nOP$ (Theorem~\ref{thm:brionmaporbit}).
Each of these families has nice combinatorial structure.

In addition, we can consider the dual Brion map on Hopf algebras by applying the Fock functor to $B^*:\pp{\P^*} \to \pp{\GP^*}$, where it suffices to restrict our domains to the primitives in light of the Cartier--Milnor--Moore theorem, in Section~\ref{sec:DualBrionMap}. From our inclusion maps, we can restrict the images to $\pp{\nPI^*}$, $\pp{\nA^*}$, and $\pp{\nOP^*}$. We describe the dual Brion map involving these three Hopf algebras in Theorem~\ref{thm:dualbrionmap}.

\section{Background on Hopf monoids}\label{sec:Background}

In this section we introduce necessary context regarding Hopf monoids, duality, and the Cartier--Milnor--Moore Theorem.

\subsection{Hopf Monoids}\label{subsec:hopfmonoids}

We first define set and vector species, which form the foundation for a Hopf monoid.
Let $\mathbb{K}$ be a field of characteristic $0$, a \Def{vector species $\mathbf{H}$} (resp. \Def{set species $\mathrm{H}$}) assigns a $\mathbb{K}$-vector space $\mathbf{H}[I]$ (resp. a set $\mathrm{H}[I]$) to each finite set $I$, and a linear map $\mathbf{H}[\sigma]:\mathbf{H}[I]\to \mathbf{H}[J]$ (resp. a map $\mathrm{H}[\sigma]:\mathrm{H}[I] \to \mathrm{H}[J]$) to each bijection $\sigma:I\to J$ between two finite sets such that our collection of maps have the property that $\mathbf{H}[\operatorname{id}]=\operatorname{id}$ and $\mathbf{H}[\sigma\circ\tau]=\mathbf{H}[\sigma]\circ\mathbf{H}[\tau]$.
We say a vector species $\mathbf{H}$ is \Def{finite-dimensional} if for every finite set $I$, the space $\mathbf{H}[I]$ is finite dimensional.
A \Def{morphism of vector species} $f:\mathbf{H} \to \mathbf{K}$ is a collection of maps $f_I: \mathbf{H}[I] \to \mathbf{K}[I]$ which satisfy the naturality axiom: for each bijection $\sigma: I \to J$, we have $f_J \circ \mathbf{H}[\sigma] = \mathbf{K}[\sigma] \circ f_I$.

With the structure above we can now construct a Hopf monoid. 
A \Def{connected Hopf monoid in vector species} $\mathbf{H}$ is a vector species $\mathbf{H}$ with $\mathbf{H}[\emptyset] = \mathbb{K}$ endowed with two families of linear maps $\mu_{S,T}: \mathbf{H}[S] \otimes \mathbf{H}[T] \longrightarrow  \mathbf{H}[I]$ and $\Delta_{S,T}: \mathbf{H}[I] \longrightarrow \mathbf{H}[S] \otimes \mathbf{H}[T]$ for each decomposition $I=S \sqcup T$, where these maps satisfy naturality, unitality, associativity, and compatibility \cite{AguiarMahajan,AguiarArdila}.
We can similarly define a \Def{connected Hopf monoid in set species} $\mathrm{H}$ which is a set species $\mathrm{H}$ such that $\mathrm{H}[\emptyset]$ is a singleton  endowed with two families of linear maps $\mu_{S,T}: \mathrm{H}[S] \times \mathrm{H}[T] \longrightarrow  \mathrm{H}[I]$ and $\Delta_{S,T}: \mathrm{H}[I] \longrightarrow \mathrm{H}[S] \times \mathrm{H}[T]$ for each decomposition $I=S \sqcup T$ satisfying the same axioms as for Hopf monoids in vector species \cite{AguiarArdila}.
Note that we notate a Hopf monoid in vector species (or a vector species) in bold, while Hopf monoids in set species (or set species) are unbolded. 
We will mainly focus on Hopf monoids in vector species, but set species will be important when we discuss the Fock functor and duality, however we will first introduce more definitions relevant to Hopf monoids.

We call the collection of the $\mu$-maps the \Def{product} and the collection of the $\Delta$-maps the \Def{coproduct} of $\mathbf{H}$. 
A \Def{symmetric braiding} of a Hopf monoid $\mathbf{H}[I]$ is a natural isomorphism $\beta_{S,T}: \mathbf{H}[S] \otimes \mathbf{H}[T] \to \mathbf{H}[T] \otimes \mathbf{H}[S]$ \cite[section 1.1.2]{AguiarMahajan}.
In all our applications, $\beta$ swaps or ``braids'' the $\mathbf{H}[S]$ and $\mathbf{H}[T]$ components of a tensor.

Let $x \in \mathbf{H}[S]$, $y \in \mathbf{H}[T]$, and $z \in \mathbf{H}[I]$ for some decomposition $I = S \sqcup T$.
Intuitively and in the many contexts of this paper, the product will describe how to merge the combinatorial structures $x$ with labels in $S$ and $y$ with labels in $T$ into one structure $x \cdot y$ with labels in $I$.
The coproduct will describe how to break up a structure $z$ with labels in $I$ into ``two" structures, the \Def{restriction} $z \vert_S$ with labels in $S$ and the \Def{contraction} $z /_S$ with labels in $T$.
A \Def{morphism of Hopf monoids} $f: \mathbf{H} \to \mathbf{K}$ is a morphism of vector species which also commutes with the products and coproducts.
We say $\mathbf{H}$ is \Def{commutative} if for all $I=S \sqcup T$ with $x \in \mathbf{H}[S]$ and $y \in \mathbf{H}[T]$, it holds that $\mu_{S,T}(x \otimes y) = \mu_{T,S}(y \otimes x)$, and \Def{cocommutative} if $\Delta_{S,T}(z) = \beta_{T,S}(\Delta_{T,S}(z))$.

In some settings it is useful to consider Hopf algebras, which in our case can be obtained from combinatorial Hopf monoids by ``forgetting'' the labelings of objects by identifying those with the same cardinality. 
We will see that this gives insight into more global structure.
To get a Hopf algebra from a Hopf monoid we use the \Def{Fock functor} $\mathcal{F}$ which is a functor from the category of connected Hopf monoids in vector species to the category of graded Hopf algebras.
Aguiar and Mahajan \cite{AguiarMahajan} give four Fock functors but we will explain the one used in this paper and by Aguiar and Ardila \cite{AguiarArdila} next.
First note that for $[n]:=\{1,2,\ldots,n\}$, the symmetric group $S_n$ acts on $\mathbf{H}[n]:=\mathbf{H}[[n]]$.
Generally, define the Fock functor on $\mathbf{H}$ as
\[ 
    \mathcal{F}(\mathbf{H}):=\bigoplus_{n \geq 0} \mathbf{H}[n]_{S_n} = \bigoplus_{n \geq 0} \mathbf{H}[n]/\mathrm{span}\{w \cdot h - h : w \in S_n, h \in \mathbf{H}[n]\}
\]
where $\mathbf{H}[n]_{S_n}$ is the space of $S_n$-coinvariants of $\mathbf{H}[n]$ for the symmetric group $S_n$ \cite{AguiarMahajan}.

We are able to further describe the Fock functor since all the Hopf monoids in vector species we consider can be seen as linearizations of set species. 
Consider the \Def{linearization functor} $\mathsf{Set} \to \mathsf{Vec}$ from the category of sets to the category of vector spaces which sends a set to the vector space with formal basis elements given by elements of the set. 
Aguiar and Mahajan \cite[Section 8]{AguiarMahajan} \cite[Section 4]{AguiarMahajan2} show that applying this functor to a set species gives a vector species. 
Moreover, this functor takes a Hopf monoid in set species $\mathrm{H}$ to a Hopf monoid in vector species $\mathbf{H}$ where $\mathbf{H}[I]$ is the vector space with basis $\mathrm{H}[I]$. 
Now for a connected Hopf monoid in vector species $\mathbf{H}$ which is the linearization of $\mathrm{H}$, the Fock functor gives the graded Hopf algebra
\[
    \mathcal{F}(\mathbf{H}) = \bigoplus_{n \geq 0} \mathrm{span}\{\text{isomorphism classes of elements of } \mathrm{H}[I] \text{ for } |I|=n \}
\]
where elements $h_1 \in \mathrm{H}[I_1]$ and $h_2 \in \mathrm{H}[I_2]$ are said to be \Def{isomorphic} if there exists a bijection $\sigma:I_1 \to I_2$ such that $\sigma(h_1):=\mathrm{H}[\sigma](h_1)=h_2$ \cite{AguiarArdila}.

\subsection{Duality of Hopf Monoids}
We will now explain dual Hopf monoids in vector species, following the definitions of Aguiar and Mahajan \cite{AguiarMahajan,AguiarMahajan2}.

Given a vector species $\mathbf{H}$, the \Def{dual vector species} $\mathbf{H}^*$ is the vector species where for a finite set $I$, $\mathbf{H}^*[I]$ is the vector space dual to $\mathbf{H}[I]$, i.e. $\mathbf{H}[I]^*$, and for any bijection $\sigma:I \to J$, the map $\mathbf{H}^*[\sigma]:\mathbf{H}^*[I]\to\mathbf{H}^*[J]$ is given by $\mathbf{H}[\sigma^{-1}]^*$, i.e. the linear dual of the map $\mathbf{H}[\sigma^{-1}]$ by precomposition of $\mathbf{H}[\sigma^{-1}]$.

If we assume that a vector species $\mathbf{H}$ is the linearization of a Hopf monoid in set species $\mathrm{H}$, then the vector space $\mathbf{H}[I]^*$ has a canonical basis which we explain next. 
We know that $\mathbf{H}[I]$ has basis given by the set $\mathrm{H}[I]$, there is a dual basis for $\mathbf{H}^*[I]$ consisting of $b^*: \mathbf{H}[I] \to \mathbb{K}$ for each basis element $b\in\mathbf{H}[I]$ where $b^*(c)$ is $1$ if $c=b$ and $0$ otherwise.

For the remainder of this section on duality, we now assume that all vector species are finite dimensional.
In particular if a vector species $\mathbf{H}$ is finite dimensional, so is its dual $\mathbf{H}^*$.
Aguiar and Mahajan show that the duality of a vector species can be seen as a functor between the opposite category of vector species $\mathsf{Sp}^{\mathrm{op}}$ and the category of vector species $\mathsf{Sp}$.
Moreover, they show that applying the duality to a Hopf monoid $\mathbf{H}$ yields another Hopf monoid in vector species called the \Def{dual Hopf monoid} $\mathbf{H}^*$ in the dual vector species.
We will next explain the product and coproduct of a dual Hopf monoid.
The product of $\mathbf{H}^*$, denoted $\mu_{S,T}^*$, is given by the composition
\[
    \mathbf{H}^*[S] \otimes \mathbf{H}^*[T] = \mathbf{H}[S]^* \otimes \mathbf{H}[T]^* \xrightarrow{c} (\mathbf{H}[S] \otimes \mathbf{H}[T])^* \xrightarrow{(\Delta_{S,T})^*} \mathbf{H}[I]^* = \mathbf{H}^*[I],
\]
where the first arrow is canonical and an isomorphism since the vector spaces are finite dimensional, and where $(\Delta_{S,T})^*$ is the linear dual of the map $\Delta_{S,T}:\mathbf{H}[I] \to \mathbf{H}[S] \otimes \mathbf{H}[T]$ for decompositions $I = S \sqcup T$.
The coproduct of $\mathbf{H}^*$, denoted $\Delta_{S,T}^*$, is given by the composition
\[
    \mathbf{H}^*[I] \xrightarrow{(\mu_{S,T})^*} (\mathbf{H}[S] \otimes \mathbf{H}[T])^* \xrightarrow{c^{-1}} \mathbf{H}[S]^* \otimes \mathbf{H}[T]^* = \mathbf{H}^*[S] \otimes \mathbf{H}^*[T].
\]
where $(\mu_{S,T})^*$ is the linear dual of the map $\mu_{S,T}:\mathbf{H}[S] \otimes \mathbf{H}[T] \to \mathbf{H}[I]$ for decompositions $I = S \sqcup T$.

Since we mainly consider the dual Hopf monoid of a Hopf monoid in vector species which is the linearization of a Hopf monoid in set species, we can compute the product and coproduct more easily using basis elements. The following is outlined by Aguiar and Mahajan  \cite[Section 8.6]{AguiarMahajan} \cite[Section 4.5]{AguiarMahajan2} but we give a proof connecting their definitions here.

\begin{prop}\label{prop:dualprodcoprod}
    Let $\mathbf{H}$ be a Hopf monoid which is the linearization of the Hopf monoid in set species $\mathrm{H}$ and finite dimensional. 
    Suppose $I=S \sqcup T$ is some decomposition and that $x \in \mathbf{H}[S]$, $y \in \mathbf{H}[T]$, and $z \in \mathbf{H}[I]$ are basis elements. 
    The product $\mu_{S,T}^*$ and coproduct $\Delta_{S,T}^*$ of the dual Hopf monoid $\mathbf{H}^*$ on the dual basis elements is given by
    \[
        \mu_{S,T}^*(x^* \otimes y^*) = \sum_{\Delta_{S,T}(z) = x \otimes y} z^*, \qquad 
        \Delta^*_{S,T}(z^*) = \sum_{\mu_{S,T}(x \otimes y) = z} x^* \otimes y^*.
    \] 
    Consequently, the product and coproduct of $\mathbf{H}^*$ for arbitrary elements is also determined.
\end{prop}
\begin{proof}
    We will first consider the product of $\mathbf{H}^*$. 
    Suppose $x \in \mathbf{H}[S]$, $y \in \mathbf{H}[T]$ are basis elements.
    We know that $\mathbf{H}[I]$ has a basis given by the set $\mathrm{H}[I]$ and there is an associated canonical dual basis for $\mathbf{H}^*[I]$.
    By definition,
    \[
        \mu_{S,T}^*(x^* \otimes y^*) = (\Delta_{S,T})^* \circ (x \otimes y)^* = (x \otimes y)^* \circ \Delta_{S,T}.
    \]
    Let $z \in \mathbf{H}[I]$ be any basis element, then $(x \otimes y)^* \circ \Delta_{S,T} (z)$ gives $1$ if $\Delta_{S,T} (z)= x \otimes y$ and $0$ otherwise.
    Hence 
    \[
        (x \otimes y)^* \circ \Delta_{S,T} (z) = \sum_{\Delta_{S,T}(z) = x \otimes y} z^*
    \]
    since all $z^*$ gives the basis elements of $\mathbf{H}^*[I]$, proving the first equality.
    Now for the coproduct of $\mathbf{H}^*$, let $z \in \mathbf{H}[I]$ be a basis element of $\mathbf{H}[I]$.
    This gives another basis element $z^* \in \mathbf{H}^*[I]$. 
    The definition of the coproduct gives
    \[
        \Delta^*_{S,T}(z^*) = c^{-1} \circ (\mu_{S,T})^* \circ z^* = c^{-1} \circ (z^* \circ \mu_{S,T}).
    \]
    Since we have a basis of with elements of the form $(x \otimes y)^*$ for $(\mathbf{H}[S] \otimes \mathbf{H}[T])^*$  where $x \in \mathbf{H}[S]$ and $y \in \mathbf{H}[T]$ are basis elements, the map $z^* \circ \mu_{S,T}: \mathbf{H}[S] \otimes \mathbf{H}[T] \to \mathbb{K}$ is determined by where it sends $x \otimes y$.
    In particular $z^* \circ \mu_{S,T}(x \otimes y)$ is $1$ if $\mu_{S,T}(x \otimes y)=z$ and is $0$ otherwise, so
    \[
        c^{-1} \circ (z^* \circ \mu_{S,T}) = c^{-1} \circ \sum_{\mu_{S,T}(x \otimes y)=z} (x \otimes y)^* = \sum_{\mu_{S,T}(x \otimes y)=z} c^{-1}((x \otimes y)^*) = \sum_{\mu_{S,T}(x \otimes y)=z} x^* \otimes y^*.
    \]
    To prove the last claim, we explain only the product as the coproduct is similar. 
    Elements in $\mathbf{H}[S]^* \otimes \mathbf{H}[T]^*$ are linear combinations of $x^* \otimes y^*$ for $x \in \mathbf{H}[S]$, $y \in \mathbf{H}[T]$ basis elements.
    By applying linearity we can apply our equation to compute the product.
\end{proof}

As a consequence of Proposition~\ref{prop:dualprodcoprod}, we obtain the following known relationship between a Hopf monoid and its dual which plays a key role in our paper.
\begin{prop}\label{thrm:dualcom}
    Let $\mathbf{H}$ be a finite dimensional Hopf monoid in vector species that arrises from the linearization of a set species $\mathrm{H}$.
    If $\mathbf{H}$ is commutative (resp. cocommutative), then $\mathbf{H}^*$ is cocommutative (resp. commutative).
\end{prop}

Moreover, a connected Hopf monoid $\mathbf{H}$ results in a connected dual Hopf monoid $\mathbf{H}^*$.
While we always assume our Hopf monoids are connected in this paper, in general a Hopf monoid $\mathbf{H}$ does not need the assumption that $\mathbf{H}[\emptyset] = \mathbb{K}$ as we defined. 
More broadly, a Hopf monoid is \Def{connected} if $\mathbf{H}[\emptyset] \cong \mathbb{K}$.
Since the Hopf monoids we consider have $\mathbf{H}[\emptyset] = \mathbb{K}$, then by duality $\mathbf{H}[\emptyset] = \mathbb{K}^* \cong \mathbb{K}$, proving the dual is connected.  

Duality also gives a relationship between the primitive and indecomposable elements, which can be viewed as the ``basic building blocks'' of Hopf monoids.
The element $z\in\mathbf{H}[I]$ is \Def{primitive} if for any nontrivial decomposition $S \sqcup T = I$, we have $\Delta_{S,T}(z)=0$; it is \Def{indecomposable} if $\mu_{S,T}(x \otimes y) = z$ implies that $S$ or $T$ is the empty set, i.e. $x=1$ or $y=1$ respectively.
We denote the set of primitives of $\mathbf{H}$ by $\mathcal{P}(\mathbf{H})$.
We use the following result about finite-dimensional Hopf monoids \cite[Section 5.6]{AguiarMahajan2}:

\begin{prop}\label{thrm:dualpi}
    Let $\mathbf{H}$ be a finite dimensional Hopf monoid. 
    If $z \in \mathbf{H}[I]$ is primitive then $z^* \in \mathbf{H}^*[I]$ is indecomposable, and if $z \in \mathbf{H}[I]$ is indecomposable then $z^* \in \mathbf{H}^*[I]$ is primitive.
\end{prop}

The proposition above can also be seen as a consequence of Proposition~\ref{prop:dualprodcoprod} for the Hopf monoids we consider. 

\subsection{The Cartier--Milnor--Moore Theorem}
\label{subsec:CCM}
    We will now explain the additional structure Hopf monoids and its primitive elements have.
    A \Def{Lie monoid in the category of vector species} is a pair $(\mathbf{L},\gamma)$ of a vector species and a family of functions called the \Def{Lie bracket} where $\gamma_{S,T}: \mathbf{L}[S] \otimes \mathbf{L}[T] \to \mathbf{L}[I]$ for each $S \sqcup T = I$ satisfy the following properties \cite{AguiarMahajan}:
    \begin{itemize}
        \item \Def{antisymmetry}, meaning that $\gamma_{S,T} + \gamma_{T,S} \circ \beta_{S,T} = 0$, and 
        \item the \Def{Jacobi identity}, which states that $\gamma_{} \circ (\gamma_{} \otimes \mathrm{id}) \circ (\mathrm{id} + \xi + \xi^2) = 0$.
        \end{itemize}
    Here, $\beta$ is the symmetric braiding map, and $\xi$ is the rotation
        \begin{center}
            \begin{tikzcd}
                {\mathbf{L}[R] \otimes \mathbf{L}[S] \otimes \mathbf{L}[T]} \arrow[rr, "{\mathrm{id} \otimes \beta_{S,T}}"] &  & {\mathbf{L}[R] \otimes \mathbf{L}[T] \otimes \mathbf{L}[S]} \arrow[rr, "{\beta_{R,T} \otimes \mathrm{id}}"] &  & {\mathbf{L}[T] \otimes \mathbf{L}[R] \otimes \mathbf{L}[S].}
            \end{tikzcd}
        \end{center}

\begin{theorem}[Proposition 1.26, \cite{AguiarMahajan}]\label{thrm:hopfarelie}
    We can regard any Hopf monoid $\mathbf{H}$ in vector species as a Lie monoid where $\gamma$ is the commutator
    $\gamma_{S,T}(x \otimes y) = \mu_{S,T}(x \otimes y) - \mu_{T,S}(y \otimes x) $
    for $x \in \mathbf{H}[S], y \in \mathbf{H}[T]$.
\end{theorem}

\begin{theorem}[Section 11.9.2, \cite{AguiarMahajan}]\label{thm:primitivesareliesub}
    For any Hopf monoid $\mathbf{H}$ in vector species, the primitives $\mathcal{P}(\mathbf{H})$ form a Lie submonoid under the commutator.
\end{theorem}

Due to the previous theorem, one can regard the primitives as a functor $\mathcal{P}: \mathsf{Hopf(Sp)} \to \mathsf{Lie(Sp)}$ from the category of Hopf monoids in vector species to the category of Lie monoids in vector species.
Moreover, since we only deal with connected Hopf monoids $\mathbf{H}$, applying the primitive functor results in a Lie monoid in \Def{positive species}, i.e. $\mathcal{P}(\mathbf{H})[\emptyset] = 0$ \cite[Section 11.9]{AguiarMahajan}.
Since our vector spaces are over a field $\mathbb{K}$ of characteristic $0$, Aguiar and Mahajan~\cite[Section 15.6.1]{AguiarMahajan} show that the primtives and our Fock functor commute, $\mathcal{F}(\mathcal{P}(\mathbf{H})) = \mathcal{P}(\mathcal{F}(\mathbf{H}))$, where this can be seen as an equality of graded Lie algebras. 

We now turn towards stating the Cartier--Milnor--Moore theorem, which allows us to recover a cocommutative Hopf monoid from its primitives. 
The \Def{free monoid} on a vector species $\mathbf{p}$ is $\mathcal{T}(\mathbf{p})$ where for a finite set $I$,
\[
    \mathcal{T}(\mathbf{p})[I]=\bigoplus_{I_1\sqcup \cdots\sqcup I_k = I} \mathbf{p}(I_1)\otimes\cdots\otimes\mathbf{p}(I_k),
\]
\cite[Section 6]{AguiarMahajan2}.
If $\mathbf{p}$ is a Lie monoid, its \Def{universal enveloping monoid} $\mathcal{U}(\mathbf{p})$ is a connected Hopf monoid obtained as the quotient of the free monoid by the ideal generated by elements of the form $$x\cdot y - y\cdot x -\gamma_{S,T}(x\otimes y)$$ for $x\in\mathbf{p}[S]$ and $y\in\mathbf{p}[T]$, \cite[Section~8.2]{AguiarMahajan2}.
We can see the universal enveloping monoid as a functor $\mathcal{U}:\mathsf{Lie(Sp)} \to \mathsf{Hop^o}$ from the category of Lie monoids in vector species to the category of connected Hopf monoids in vector species.

The following theorem tells when the universal enveloping monoid can be used to reconstruct a Hopf monoid from its Lie monoid of primitives.

\begin{theorem}[Cartier--Milnor--Moore \cite{Cartier2007,MilnorMoore}, Hopf monoid version \protect{\cite[Theorem 120]{AguiarMahajan2}}]\label{thrm:CMM}
The functors
\[
    \begin{tikzcd}
            \mathsf{Lie(Sp^+)} \arrow[r, shift left=1ex, "\mathcal{U}"{name=G}] & \mathsf{coHopf^o}\arrow[l, shift left=.5ex, "\mathcal{P}"{name=F}],
        \end{tikzcd}
\]
where $\mathsf{Lie(Sp^+)}$ is the category of Lie monoids in positive species and $\mathsf{coHopf^o}$ is the category of cocommutative connected Hopf monoids in vector species, form an adjoint equivalence. 
In other words, suppose $\mathbf{H}$ is a connected cocommutative Hopf monoid in vector species, then $\mathcal{U}(\mathcal{P}(\mathbf{H}))$ is isomorphic to $\mathbf{H}$.
In particular, the inclusion of $\mathcal{P}(\mathbf{H})$ into $\mathbf{H}$ extends into an isomorphism of Hopf monoids $\varphi: \mathcal{U}(\mathcal{P}(\mathbf{H})) \to\mathbf{H}$.
\end{theorem}

We will see that the connected Hopf monoid $\GP$ (and hence its submonoids) is commutative, hence the connected dual Hopf monoid is cocommutative by Theorem~\ref{thrm:dualcom}, and Theorem~\ref{thrm:CMM} applies.

\section{Our examples of Hopf monoids}\label{sec:hopfmonoidexamples}

We will now define the Hopf monoids that we consider in this paper given by Aguiar and Ardila~\cite{AguiarArdila}, namely $\mathbf{P}$, $\nGP$, and its submonoids $\nPI$, $\nA$, and $\nOP$ involving permutahedra, associahedra, and orbit polytopes respectively.

\subsection{The Hopf Monoid of Posets}\label{subsec:HopfofPosets}

If $p$ is a poset structure on a finite set $I$ and $I = S \sqcup T$, then the \Def{restriction of $p$ to $S$}, or $p\vert_S$, consists of the relations in $p$ on $S$ that only relate elements elements of $S$.
We also say $S$ is a \Def{lower set} of the poset $p$ when no element of $T$ is less than an element of $S$.
In other words, if $s \in S$ and $x \leq s$ for some $x \in I$, then $x \in S$. 

Given a finite set $I$, let $\mathrm{P}[I]$ denote the set of all posets on $I$.
Moreover, let $\P[I]$ be the vector space with basis indexed by poset structures on $I$, i.e. the linearization of $\mathrm{P}$.
The relabeling maps in this vector species are induced by relabeling the posets. 
The product $\mu$ and coproduct $\Delta$ for a $I = S \sqcup T$ be a decomposition on $\P$ are defined as follows:
\begin{itemize}
    \item the product
    $$\mu_{S,T}: \P[S] \otimes \P[T] \to \P[I],$$
    $$\mu_{S,T}(p_1 \otimes p_2) = p_1 \sqcup p_2.$$
    This is the disjoint union of the posets, i.e. $p_1 \sqcup p_2$ is a poset on $I$ with no relations between elements of $S$ and $T$. 
    
    \item the coproduct
    $$\Delta_{S,T}: \P[I] \to \P[S] \otimes \P[T],$$
    $$\Delta_{S,T}(p) = 
    \begin{cases}
        p\vert_S \otimes p\vert_T & \text{if } S \text{ is a lower set of } p, \\
        0 & \text{otherwise.}
    \end{cases}$$
\end{itemize}

\begin{example}[Coproduct in $\P$]
Let us coproduct of a poset with respect to two different decompositions:
\begin{figure}[H]
    \centering
    \begin{tikzpicture}[
    vertex/.style={draw=black,fill=black,circle,inner sep=1.5pt},
    redvertex/.style={draw=nred,fill=nred,circle,inner sep=1.5pt}
    ]
    \node at (-2.9,-.5) {\Large$\Delta_{{\color{nred}\{c,d\}},\{a,b\}}$\Huge$($};
    \coordinate[vertex] (b) at (-1,0);
    \coordinate[vertex] (a) at (-1.3,-.5);
    \coordinate[redvertex] (c) at (-.7,-.5);
    \coordinate[redvertex] (d) at (-1,-1);
    \node[anchor=south] at (b) {$b$};
    \node[anchor=east] at (a) {$a$};
    \node[anchor=west] at (c) {\color{nred}$c$};
    \node[anchor=north] at (d) {\color{nred}$d$};
    \draw[thick] (b)--(a);
    \draw[thick] (b)--(c);
    \draw[thick] (a)--(d);
    \draw[thick,nred] (c)--(d);
    \node at (0,-.5) {{\Huge$)$}\;$=$};
    
    \coordinate[redvertex] (c') at (1-0.4,-.25);
    \coordinate[redvertex] (d') at (1-0.4,-.75);
    \node[anchor=south] at (c') {\color{nred}$c$};
    \node[anchor=north] at (d') {\color{nred}$d$};
    \draw[thick,nred] (c')--(d');
    \node at (1.5-0.4,-.5){\Large$\otimes$};
    
    \coordinate[vertex] (b') at (2-0.4,-.25);
    \coordinate[vertex] (a') at (2-0.4,-.75);
    \node[anchor=north] at (a') {$a$};
    \node[anchor=south] at (b') {$b$};
    \draw[thick] (a')--(b');
\end{tikzpicture}
\hspace{3em}
\begin{tikzpicture}[xshift=-2.9,yshift=3.5,
    vertex/.style={draw=black,fill=black,circle,inner sep=1.5pt},
    redvertex/.style={draw=nred,fill=nred,circle,inner sep=1.5pt}
    ]
    \node at (-2.9,-3.5) {\Large$\Delta_{{\color{nred}\{b,c\}},\{a,d\}}$\Huge$($};
    \coordinate[redvertex] (b) at (-1,-3);
    \coordinate[vertex] (a) at (-1.3,-3.5);
    \coordinate[redvertex] (c) at (-.7,-3.5);
    \coordinate[vertex] (d) at (-1,-4);
    \node[anchor=south] at (b) {\color{nred}$b$};
    \node[anchor=east] at (a) {$a$};
    \node[anchor=west] at (c) {\color{nred}$c$};
    \node[anchor=north] at (d) {$d$};
    \draw[thick] (b)--(a);
    \draw[thick,nred] (b)--(c);
    \draw[thick] (a)--(d);
    \draw[thick] (c)--(d);
    \node at (0.1,-3.5) {{\Huge$)$}\;$=0.$};
    
    

\end{tikzpicture}
    \label{fig:poset_coproduct}
\end{figure}
\noindent In the first expression, $S = \{c,d\}$ is a lower set because no element in $T=\{a,b\}$ is less than any element of $S$, so we can express the coproduct as a tensor.
In the second expression, $\{b,d\}$ is not a lower set so the coproduct is $0$.
\end{example}

The Hopf monoid $\mathbf{P}$ is commutative but not cocommutative.

\subsection{The Hopf monoid of Generalized Permutahedra}\label{subsec:gp}

Before describing the Hopf monoid structure, we will explain the relevant geometric background for generalized permutahedra.

\subsubsection{Geometric background}\label{subsub:geobackground}
Let $I$ be a finite set and $\mathbb{R}^I$ be the real vector space whose vectors are $I$-tuples of real numbers. 
In other words, $\mathbb{R}^I = \{(a_i)_{i \in I} : a_i \in \mathbb{R} \}$. 
Let $\{e_i : i \in I\}$ be the standard basis for $\mathbb{R}^I$.
We also have the standard inner product for $\mathbb{R}^I$ where $\langle x,y\rangle = \sum_{i \in I}x_i y_i$ for $x,y \in \mathbb{R}^I$.
We denote the dual of $\mathbb{R}^I$ by $(\mathbb{R}^I)^*$ where $(\mathbb{R}^I)^* = \{y: \mathbb{R}^I \to \mathbb{R} : y \text{ is linear} \}$. 
Such a $y \in (\mathbb{R}^I)^*$ is called a \Def{linear functional}.

A \Def{polyhedron} $Q \subseteq \mathbb{R}^I$ is the set of solutions to a finite system of linear inequalities.
A \Def{polytope} is a bounded polyhedron, or alternatively the convex hull of a finite set of points in $\mathbb{R}^I$, where the \Def{convex hull} of $v_1,\ldots,v_n \in \mathbb{R}^I$ is 
\[
    \operatorname{conv}\left(v_1,\ldots,v_n\right) = \left\lbrace x \in \mathbb{R}^I : \exists \lambda_i \in \mathbb{R}, x=\lambda_1 v_1 + \cdots + \lambda_n v_n, \sum_{i=1}^n \lambda_i = 1, \lambda_i \geq 0  \right\rbrace.
\]
We can also define a polyhedral \Def{cone} $C \subseteq \mathbb{R}^I$ as a polyhedron with the property $\mu x \in C$ for any $x \in C$ and real number $\mu \geq 0$.
If we are given points $w_1,\ldots,w_n \in \mathbb{R}^I$ the \Def{cone generated} by these points is
\[
    \cone(w_1,\ldots,w_n) := \mathbb{R}_{\geq 0} w_1 + \cdots + \mathbb{R}_{\geq 0} w_n \subseteq \mathbb{R}^I.
\]

We will now discuss faces of polyhedra.
An \Def{admissible hyperplane} of a polyhedron $Q \subseteq \mathbb{R}^I$ is a hyperplane $H=\{x \in \mathbb{R}^I : \langle a,x \rangle = b\} \subseteq \mathbb{R}^I$ for some $a\in \mathbb{R}^I \setminus \{0\}$ and $b \in \mathbb{R}$ such that $\langle a,x \rangle \leq b$ for all $x \in Q$.
A \Def{face} of a polyhedron $Q \subseteq \mathbb{R}^I$ is $Q \cap H$ where $H$ is an admissible hyperplane.
We also define the polyhedron itself as a face and the empty set as a face and say a face is \Def{proper} face if it is nonempty nor the entire polyhedron. 
Moreover, nonempty faces of a polyhedron $Q \subseteq \mathbb{R}^I$ can also be seen as maximizing linear functionals which we explain next.
Proper faces $F$ are equal to $Q \cap H$ for some admissible hyperplane $H = \{x \in \mathbb{R}^I : \langle a,x \rangle = b\} \subseteq \mathbb{R}^I$ where $a\in \mathbb{R}^I \setminus \{0\}$ and $b \in \mathbb{R}$. 
We can consider $\langle a , \text{--} \rangle$ as a linear functional in $(\mathbb{R}^I)^*$, hence the points in $Q$ achieving the maximal value of  $\langle a , \text{--} \rangle$ are exactly the points in $F$. 
Since various hyperplanes can define $F$, there can be various such linear functionals.
Moreover, $Q$ itself can be seen as the points in $Q$ maximizing the linear functional given by the zero map.

Motivated by the face structure, we will see that there is an associated fan for polyhedra.
\begin{definition}
A \Def{fan} $\mathcal{F}$ is a collection of cones where,
\begin{itemize}
    \item If $C\in \mathcal{F}$ then all the faces of the cone $C$ are in $\mathcal{F}$
    \item For $C,D\in \mathcal{F}$, if $C\cap D\neq \emptyset$ then $C\cap D$ is a face of both $C$ and $D$.
\end{itemize}
We say a fan $\mathcal{G}$ is a \Def{coarsening} of another fan $\mathcal{F}$ (or $\mathcal{F}$ is a refinement of $\mathcal{G}$) if every cone of $\mathcal{F}$ is contained in a cone of $\mathcal{G}$.
\end{definition}

\begin{definition}
    Given a nonempty face $F$ of a polytope $P \subseteq \mathbb{R}^I$, the \Def{normal cone of $F$} is 
    $$
        \mathcal{N}_P(F) = \{y \in (\mathbb{R}^I)^* \mid y \text{ attains its maximal value on } P \text{ at every point in } F \}.
    $$
    In other words, $\mathcal{N}_P(F)$ is the cone of linear functionals $y$ such that the $y$-maximal face of $P$ contains the face $F$. 
    
    The \Def{normal fan} $\mathcal{N}_P$ is the fan consisting of the normal cones $\mathcal{N}_P(F)$ for every face of the polytope $P$.
    This is indeed a fan \cite{AguiarArdila}.
    
    We say that two polytopes are \Def{normally equivalent} if they have the same normal fan. Given a polytope $P$, the collection of all normally equivalent polytope is called the \Def{normal equivalence class} of $P$.
\end{definition}
See Figure~\ref{fig:normalfanequivalence} for an example of a geometric realization of the normal fan of a polytope and two normally equivalent polytopes.

\begin{figure}
     \centering
\begin{tikzpicture}[ every node/.style={scale=0.85},
	edge/.style={thick, color=nblue},
	facet/.style={fill=nblue,fill opacity=0.300000}]

\coordinate (A) at (0,2);
\coordinate (B) at (1,1);
\coordinate (C) at (1,-1);
\coordinate (D) at (-1,-1);
\coordinate (E) at (-1,1);


\draw[edge] (A) node[right]{} -- (B) node[left]{} -- (C) node[right]{} -- (D) node[right]{} -- (E) node[right]{} --cycle;

\fill[facet] (A) -- (B) -- (C) -- (D) -- (E) -- cycle;

\end{tikzpicture} \qquad
\begin{tikzpicture}[ every node/.style={scale=0.85},
	edge/.style={thick, color=nblue},
	facet/.style={fill=nblue,fill opacity=0.300000}]

\coordinate (A) at (0,2);
\coordinate (B) at (1.5,0.5);
\coordinate (C) at (1.5,-0.5);
\coordinate (D) at (-1,-0.5);
\coordinate (E) at (-1,1);


\draw[edge] (A) node[right]{} -- (B) node[left]{} -- (C) node[right]{} -- (D) node[right]{} -- (E) node[right]{} --cycle;

\fill[facet] (A) -- (B) -- (C) -- (D) -- (E) -- cycle;

\end{tikzpicture}
     \begin{tikzpicture}[scale=0.9, every node/.style={scale=0.85},
	edge/.style={thick, color=nblue},
	facet/.style={fill=nblue,fill opacity=0.300000}]

\coordinate (A) at ({-sqrt(2)},{sqrt(2)});
\coordinate (B) at ({sqrt(2)},{sqrt(2)});
\coordinate (C) at (-2,0);
\coordinate (D) at (2,0);
\coordinate (E) at (0,-2);

\shade[shading=radial, inner color=lgray]
(0,0) circle (2.3);

\draw[->,thick] (0,0) -- (A); 
\draw[->,thick] (0,0) -- (B);
\draw[->,thick] (0,0) -- (C); 
\draw[->,thick] (0,0) -- (D);
\draw[->,thick] (0,0) -- (E);



\end{tikzpicture} 
     \label{fig:nequivex}
    \caption{A pair of normally equivalent polytopes and their normal fan.}
    \label{fig:normalfanequivalence}
\end{figure}

We now introduce the families of polytopes which will play a key role in the Hopf monoid of generalized permutahedra.
For any finite set $I$, the \Def{standard permutahedron} $\pi_I$ in $\mathbb{R}^I$ is the convex hull of all points of the form $\sum_{i\in I} f(i) e_i\in\mathbb{R}^I$, where $f:I\to [n]$ is a bijection and $n=|I|$. 
See the top of Figure~\ref{fig:gpexample}.
The normal fan of $\pi_I$ is called the \Def{braid fan}, $\mathcal{N}_{\pi_I}$.
We can see this fan as the collection of faces of the \Def{braid arrangement} $\mathcal{B}_I$ in $(\mathbb{R}^I)^*$ \cite{AguiarArdila}.
This arrangement consists of $\binom{n}{2}$ hyperplanes defined by $x_i=x_j$ for $i,j \in I$ and $i \neq j$. 
    
\begin{definition}[\cite{AguiarArdila}]\label{def:gp}
    A \Def{generalized permutahedron} in $\mathbb{R}^I$ is a polytope whose normal fan is a coarsening of the braid fan $\mathcal{N}_{\pi_I}$.
    Equivalently, it is a polytope whose edge directions have the form $\{ e_i - e_j : i \neq j \}$; i.e. the edges are all parallel to those of the permutahedron.
\end{definition}
See Figure~\ref{fig:gpexample} for some examples of generalized permutahedra.

\begin{figure}
        \centering
        \tdplotsetmaincoords{70}{115}
\begin{tikzpicture}[scale=0.5, every node/.style={scale=0.75}, tdplot_main_coords,
	edge/.style={line width = 0.5pt, color=ngreen},
	facet/.style={fill=ngreen,fill opacity=0.300000}]

\coordinate (A) at (1,2,3);
\coordinate (B) at (1,3,2);
\coordinate (C) at (2,3,1);
\coordinate (D) at (3,2,1);
\coordinate (E) at (3,1,2);
\coordinate (F) at (2,1,3);

\draw[->] (0,0,0) -- (2,0,0) node[left]{$1$}; 
\draw[->] (0,0,0) -- (0,2,0) node[right]{$2$};
\draw[->] (0,0,0) -- (0,0,3) node[above]{$3$};

\draw[edge] (A) node[right]{} -- (B) node[right]{} -- (C) node[right]{} -- (D) node[below]{} -- (E) node[left]{} -- (F) node[left]{} -- cycle;

\fill[facet] (A) -- (B) -- (C) -- (D) -- (E) -- (F) -- cycle;

\end{tikzpicture}

\begin{tikzpicture}[scale=0.5, every node/.style={scale=0.75}, tdplot_main_coords,
	edge/.style={thick, color=ngreen},
	facet/.style={fill=ngreen,fill opacity=0.300000}]

\coordinate (A) at (1,2,3);
\coordinate (B) at (1,3,2);
\coordinate (C) at (2,3,1);
\coordinate (D) at (3.5,1.5,1);
\coordinate (E) at (3.5,1,1.5);
\coordinate (F) at (2,1,3);

\draw[->] (0,0,0) -- (2,0,0) node[left]{}; 
\draw[->] (0,0,0) -- (0,2,0) node[right]{};
\draw[->] (0,0,0) -- (0,0,3) node[above]{};

\draw[edge] (A) node[right]{} -- (B) node[right]{} -- (C) node[right]{} -- (D) node[below]{} -- (E) node[left]{} -- (F) node[left]{} -- cycle;

\fill[facet] (A) -- (B) -- (C) -- (D) -- (E) -- (F) -- cycle;

\end{tikzpicture}\quad\begin{tikzpicture}[scale=0.5, tdplot_main_coords,
	edge/.style={color=ngreen,thick},
	facet/.style={fill=ngreen,fill opacity=0.300000}]]

\coordinate (A) at (1,2,3);
\coordinate (B) at (1,4,1);
\coordinate (C) at (3,2,1);
\coordinate (D) at (3,1,2);
\coordinate (E) at (2,1,3);

\draw[->] (0,0,0) -- (3,0,0); 
\draw[->] (0,0,0) -- (0,3,0);
\draw[->] (0,0,0) -- (0,0,3);

\draw[edge] (A) -- (B) -- (C) -- (D) -- (E) -- cycle;

\fill[facet] (A) -- (B) -- (C) -- (D) -- (E) -- cycle;

\end{tikzpicture}\quad\begin{tikzpicture}[scale=0.5, tdplot_main_coords,
	edge/.style={color=ngreen,thick},
	facet/.style={fill=ngreen,fill opacity=0.300000}]]

\coordinate (A) at (1,2,3);
\coordinate (B) at (1,4,1);
\coordinate (C) at (4,1,1);
\coordinate (E) at (2,1,3);

\draw[->] (0,0,0) -- (3,0,0); 
\draw[->] (0,0,0) -- (0,3,0);
\draw[->] (0,0,0) -- (0,0,3);

\draw[edge] (A) -- (B) -- (C) -- (E) -- cycle;

\fill[facet] (A) -- (B) -- (C) -- (E) -- cycle;

\end{tikzpicture}
        \caption{Examples of generalized permutahedra in $\mathbb{R}^{[3]}$}
        \label{fig:gpexample}
\end{figure}

Next, we will explain how to construct special subfamilies of generalized permutahedra that will form the basis for our Hopf submonoids $\nA$ and $\nOP$.
The first subfamily of generalized permutahedra we define is {(Loday) associahedra} \cite{AguiarArdila, Loday}.
The following construction of the associahedron uses the Minkowski sum of two sets.
The \Def{Minkowski sum} of $A,B \subseteq \mathbb{R}^I$  is $A+B:=\{a+b : a \in A, b \in B\} \subseteq \mathbb{R}^I$.
For a finite set $I$ with a linear ordering $\ell$, the \Def{(Loday) associahedron} is defined as the Minkowski sum $a_\ell = \sum_{i \leq_\ell j} \Delta_{[i,j]_\ell}$ where $\Delta_{[i,j]_\ell} = \operatorname{conv}\{e_k : i \leq_\ell k \leq_\ell j\}$.
For example, suppose we have $I=[3]$ with the usual linear ordering. Then we have 
\[
    a_{[3]} = \sum_{1 \leq i \leq j \leq 3} \Delta_{[i,j]} = \Delta_{[1,1]} + \Delta_{[2,2]} + \Delta_{[3,3]} + \Delta_{[1,2]} + \Delta_{[2,3]} + \Delta_{[1,3]}
\]
which is shown in Figure~\ref{fig:associahedronSum}.
 \begin{figure}
        \centering
        \tdplotsetmaincoords{70}{150}

\begin{tikzpicture}[scale=.5, tdplot_main_coords,
	edge/.style={color=ngreen},
	facet/.style={fill=ngreen,fill opacity=0.300000}]]

\coordinate (A) at (1,2,3);
\coordinate (B) at (1,4,1);
\coordinate (C) at (3,2,1);
\coordinate (D) at (3,1,2);
\coordinate (E) at (2,1,3);

\draw[->] (0,0,0) -- (3,0,0); 
\draw[->] (0,0,0) -- (0,3,0);
\draw[->] (0,0,0) -- (0,0,3);

\draw[edge] (A) -- (B) -- (C) -- (D) -- (E) -- cycle;

\fill[facet] (A) -- (B) -- (C) -- (D) -- (E) -- cycle;

\end{tikzpicture}
\,
\begin{tikzpicture}[scale=0.5]
	\node[anchor=south] at (0,-5){$=$};
\end{tikzpicture}
\,
\begin{tikzpicture}[scale=0.6,
	edge/.style={thick,color=ngreen},
	facet/.style={fill=ngreen,fill opacity=0.300000}]]
\pgfmathsetmacro{\shifting}{11.5}

\coordinate (A) at (90:1);
\coordinate (B) at (210:1);
\coordinate (C) at (330:1);

\foreach \x in {0,2,4} {
    \node at (\x+1,0) {\Large $+$};
    \draw[color=ngreen,fill=ngreen] (\x,0) circle (2pt);
};

\draw[edge] ([shift={(5.75,0)}] A) -- ([shift={(5.75,0)}] C);

\node at (7.5,0) {\Large $+$};

\draw[edge] ([shift={(8.75,0.05)}] C) -- ([shift={(8.75,0.05)}] B);

\node at (10,0) {\Large $+$};

\draw[edge] ([shift={(\shifting,0)}] A) -- ([shift={(\shifting,0)}] B) -- ([shift={(\shifting,0)}] C) -- cycle;

\fill[facet] ([shift={(\shifting,0)}] A) -- ([shift={(\shifting,0)}] B) -- ([shift={(\shifting,0)}] C) -- cycle;

\end{tikzpicture}
        \caption{An example of the associahedron $a_{[3]}$}
        \label{fig:associahedronSum}
\end{figure}
Another subfamily of generalized permutahedra are orbit polytopes.
Let $x \in \mathbb{R}^I$, then the \Def{orbit polytope} $\Or(x)$ is the convex hull of all permutations of the coordinates of $x$.
That is, $\Or(x) = \conv{\sigma(x) \in \mathbb{R}^I : \sigma \in S_I}$ where $S_I$ is the symmetric group on the set $I$ which acts by permuting the coordinates. 
See Figure~\ref{fig:orbitexample} for an example of an orbit polytope given by $\conv{(0,1,1),(1,0,1),(1,1,0)}$.

 \begin{figure}
    \centering
    \tdplotsetmaincoords{70}{115}

\begin{tikzpicture}[scale=1.5, every node/.style={scale=0.85}, tdplot_main_coords,
	edge/.style={thick, color=ngreen},
	facet/.style={fill=ngreen,fill opacity=0.300000}]

\coordinate (A) at (0,1,1);
\coordinate (B) at (1,0,1);
\coordinate (C) at (1,1,0);

\draw[->] (0,0,0) -- (1.5,0,0) node[left]{x}; 
\draw[->] (0,0,0) -- (0,1.5,0) node[right]{y};
\draw[->] (0,0,0) -- (0,0,1.5) node[above]{z};

\draw[edge] (A) node[right]{$(0,1,1)$} -- (B) node[left]{$(1,0,1)$} -- (C) node[right]{$(1,1,0)$} -- cycle;

\fill[facet] (A) -- (B) -- (C) -- cycle;

\end{tikzpicture}
    \caption{The orbit polytope $\Or(x)$ for $x=(0,1,1) \in \mathbb{R}^{[n]}$}
    \label{fig:orbitexample}
\end{figure}

Generalized permutahedra are polytopes, but we can extend our definition to possibly include unbounded polyhedra.
An \Def{extended generalized permutahedra} $P \subseteq \mathbb{R}^I$ is a polyhedron whose normal fan is a coarsening of a subfan of the braid fan.
An important family of extended generalized permutahedra are {poset cones}.
The \Def{poset cone} of a poset $p$ is $\cone(p) :=  \cone(e_i - e_j : i \geq j \text{ in } p )$.
In Figure~\ref{fig:VertexPosetConeExample}, the poset on the right has the poset cone given by $\cone(e_2-e_1,e_1-e_3)$.

In fact, poset cones are exactly the tangent cones of generalized permutahedra up to translation.
The \Def{tangent cone} $\cone_F(P)$ at a face $F$ of a polytope $P$ is the cone generated by all of the directions of edges with at least one endpoint in $F$, oriented to point out of $F$.
A \Def{vertex cone} is a tangent cone at a vertex $v$, denoted by $\mathrm{cone}_v(P)$.
This gives us a connection between posets and extended generalized permutahedra.
See Figure~\ref{fig:VertexPosetConeExample} for an example of a vertex cone which is also a (translated) poset cone shown in the same figure.
\begin{figure}
        \centering
        \scalebox{0.8}{ \tdplotsetmaincoords{54.75}{135}

\begin{tikzpicture}[scale=1.2, every node/.style={scale=0.85}, tdplot_main_coords,
	edge/.style={thick, color=ngreen},
	facet/.style={fill=ngreen,fill opacity=0.300000}]

\coordinate (A) at (1,2,3);
\coordinate (B) at (1,3,2);
\coordinate (C) at (2,3,1);
\coordinate (D) at (3,2,1);
\coordinate (E) at (3,1,2);
\coordinate (F) at (2,1,3);







\draw[edge] (A) node[right]{$(1,2,3)$} -- (B) node[right]{$(1,3,2)$} -- (C) node[right]{$(2,3,1)$} -- (D) node[below]{$(3,2,1)$} -- (E) node[left]{$ (3,1,2)$} -- (F) node[left]{$(2,1,3)$} -- cycle;

\fill[facet] (A) -- (B) -- (C) -- (D) -- (E) -- (F) -- cycle;

\end{tikzpicture}
\tdplotsetmaincoords{54.75}{135}
\hspace{-1em}\begin{tikzpicture}[scale=1.2, every node/.style={scale=0.85}, tdplot_main_coords,
	edge/.style={thick, color=ngreen},
	facet/.style={fill=ngreen,fill opacity=0.300000},ray/.style={color=norange}]

\coordinate (A) at (1,2,3);
\coordinate (B) at (1,3,2);
\coordinate (C) at (2,3,1);
\coordinate (D) at (3,2,1);
\coordinate (E) at (3,1,2);
\coordinate (F) at (2,1,3);







\draw[edge] (A) node[right]{$(1,2,3)$} -- (B) node[right]{$(1,3,2)$} -- (C) node[right]{$(2,3,1)$} -- (D) node[below]{$(3,2,1)$} -- (E) node[left]{$ (3,1,2)$} -- (F) node[left]{$(2,1,3)$} -- cycle;

\fill[facet] (A) -- (B) -- (C) -- (D) -- (E) -- (F) -- cycle;

\shade[top color = norange, middle color = white, opacity=0.2, shading angle = 30]
(2,1,3) -- (0,3,3) -- (2,3,1) -- (4,1,1) -- cycle;


\draw[ray,->,thick] (2,1,3) -- (1,2,3);

\node[below,norange] at (1.2,1.8,2.9) {$e_2 - e_1$};

\draw[ray,->,thick] (2,1,3) -- (3,1,2);
\node[right,norange] at (2.8,1,2.2) {$e_1 - e_3$};


\end{tikzpicture}
\begin{tikzpicture}[scale=1.2,vertex/.style={draw=black,fill=black,circle,inner sep=1.5pt}]
    \coordinate[vertex] (2) at  (0,1);
    \node[right] at (2) {$2$};
    
    \coordinate[vertex] (1) at (0,0);
    \node[right] at (1) {$1$};
    
    \coordinate[vertex] (3) at (0,-1);
    \node[right] at (3) {$3$};
    
    \node[draw,opacity=0] at (0,-1.65) {$ $};
    
    \draw[thick] (2) -- (1) -- (3);
    
    \node[left] at (-0.75,0.5) {$2 \geq 1$};
    \node[left] at (-0.75,-0.5) {$1 \geq 3$};
\end{tikzpicture}}
         \caption{The ray generators of $\cone_{(2,1,3)}(\pi_{[3]})$, and the associated poset}
     \label{fig:VertexPosetConeExample}
 \end{figure}

\begin{remark}
    There is also an related construction to poset cones called the order cone. For a poset $p$ over $I$ the \Def{order cone} is
    \[
    \mathcal{K}_p = \{ x \in \mathbb{R}_{\geq 0}^I \mid x_i \leq x_j \text{ if } i \leq j \text{ in } p\}
    \]
    Notice that the defining normals of the faces are in the directions of the form $e_i-e_j$ since the inequalities can also be expressed as $(e_i - e_j) \cdot x \leq 0$, so the order cone is dual to the poset cone.
    Since the dual cone of a vertex cone is a normal cone, for generalized permutahedra the normal fan can be seen as a collection of order cones. 
\end{remark}

\subsubsection{The Product and Coproduct of Generalized Permutahedra}\label{subsubsec:hopfGP}

We will now give the Hopf monoid structure for generalized permutahedra, given by Aguiar and Ardila \cite{AguiarArdila}.
Given a set $I$, let $\mathrm{GP}[I]$ denote the set of all generalized permutahedra that live in $\mathbb{R}^I$.
This is an infinite set, but up to normal equivalence there is only a finite number of normally equivalent generalized permutahedra.
This can be seen by the fact that each polytope corresponds to a coarsening of the braid fan, and there are only finitely many coarsened fans.
We define $\overline{\mathrm{GP}}[I]$ to denote the set of equivalence classes of normally equivalent generalized permutahedra, and $\overline{\mathrm{GP}}$ is a set species \cite{AguiarArdila}.
We can create a finite dimensional vector species $\nGP$ of generalized permutahedra up to normal equivalence by defining $\nGP[I]$ to be a real vector space with a basis indexed by generalized permutahedra that live in $\mathbb{R}^I$, i.e. the linearization of $\overline{\mathrm{GP}}$. 
The relabeling maps in this species are induced by relabeling the basis of $\mathbb{R}^I$.
The \Def{Hopf monoid of generalized permutahedra up to normal equivalence} $\nGP$ is the vector species described previously with the following product and coproduct.
Given any two generalized permutahedra $P\subseteq\R^S$ and $Q\subseteq\R^T$ for finite sets $S$ and $T$, their \Def{product} is given by
        \[
        \mu_{S,T}(P \otimes Q) = P \times Q := \{(p,q):p\in P,q\in Q\}\subseteq\R^S\times \R^T = \R^I,
        \]
which is again a generalized permutahedron.
Note that this product is commutative in the sense of Hopf monoids. 
We further have the following theorem defining the coproduct $\Delta_{S,T}(P) = P|_S \otimes P/_S$.
First, for a polytope $P\subseteq \mathbb{R}^I$ we say denote $\1_S:=\sum_{i\in S} e_i$ for $S \subseteq I$ and say the $\1_S$-maximal face of $P$ is the face of $P$ which maximizes the linear functional $\langle \1_S, \text{--} \rangle$.
\begin{theorem}[\cite{Fujishige}]\label{thrm:GP coproduct}
    Let $P\in \GP[I]$ and suppose $I=S\sqcup T$.
    Then the $\1_S$-maximal face of $P$ can be decomposed uniquely as a product of polytopes $P|_S \times P/_S$, where $P|_S\in \GP[S]$ and $P/_S\in\GP[T]$.
\end{theorem}
\begin{theorem}[\cite{AguiarArdila}]\label{thrm:GPHopf}
    $\GP$ is a commutative Hopf monoid under this product and coproduct.
\end{theorem}

There is also a Hopf monoid of extended generalized permutahedra $\GP^+$ obtained in a similarly way by considering the Hopf monoid of extended generalized permutahedra in set species ${\mathrm{GP}}^+$ and linearizing.
We have seen that posets can be represented as cones by their poset cones.
This induces a Hopf monoid structure on poset cones that realizes them as a submonoid of extended generalized permutahedra \cite[Proposition 3.4.6]{AguiarArdila}.

\begin{theorem}[\cite{AguiarArdila}]
    Identifying each poset $p$ with the corresponding poset cone $\cone(p)$ embeds $\P$ as a Hopf submonoid of extended generalized permutahedra.
\end{theorem}

\subsection{Submonoids of $\nGP$}\label{subsec:submonoidsofgp}
We now introduce the more the submonoids of $\nGP$ relating to permutahedra, associahedra, orbit polytopes.
By restricting to subfamilies of generalized permutahedra, Aguiar and Ardila \cite{AguiarArdila} and the second author \cite{Mariel} express the coproduct more specifically, which will be helpful when we apply duality later.
All these families are also linearizations of Hopf monoids in set species, but we only describe the Hopf monoid structure in vector species.

\subsubsection{The Hopf Monoid of Permutahedra}
We first consider the submonoid of $\nGP$ that results from restricting to standard permutahedra.
Specifically, define the \Def{Hopf monoid of generalized permutahedra up to normal equivalence} $\nPI$ by the following.
For $I$ a finite set, let $\nPI[I]$ is the vector space with a basis indexed by normal equivalence classes of products of standard permutahedra $\pi_{S_1}\times\dots\times\pi_{S_k}\subseteq \mathbb{R}^I$ where $S_1\sqcup\dots\sqcup S_k = I$.
The product and coproduct are induced from $\nGP$.
In particular, the coproduct on standard permutahedra $\pi_I$ satisfies
\begin{equation}\label{eq:coprodperm}
    \Delta_{S,T}(\pi_I) = \pi_S \otimes \pi_T
\end{equation} for each decomposition $I = S \sqcup T$ \cite{AguiarArdila}.


\subsubsection{The Hopf Monoid of Associahedra}
We can similarly consider the submonoid of $\nGP$ which results from restricting to associahedra. 
Define the \Def{Hopf monoid of associahedra up to normal equivalence} $\nA$ by the following.
For a given ground set $I$, define 
\[
\nA[I] = \mathrm{span}\{a_{\ell_1} \times \cdots \times a_{\ell_k} \mid \ell_i \text{ is a linear order on } S_i \text{ for } I = S_1 \sqcup \cdots \sqcup S_k\}.
\]
This again forms a submonoid of $\nGP$, where the coproduct can be expressed as  
\begin{equation}\label{eq:coproductas}
    \Delta_{S,T}(a_\ell) = a_{\ell\vert S} \otimes \left( a_{\ell \vert T_1} \times \cdots \times a_{\ell \vert T_k} \right).
\end{equation}
for each linear order $\ell$ on $I$ and any decomposition $I = S \sqcup T$ where $T = T_1 \sqcup \cdots \sqcup T_k$ is a decomposition of $T$ into maximal intervals of $\ell$ \cite{AguiarArdila}. Here $\ell \vert U$ denotes the restriction of the linear order $\ell$ to $U$. 

\begin{example}[Coproduct in $\nA$]
Consider $a_{1234} \in \overline{\mathbf{A}}[\{1,2,3,4 \}]$. Then,
    \begin{align*}
        \Delta_{\{2\}, \{1,3,4\}}(a_{1234}) &= a_{\ell\vert S} \otimes \left( a_{\ell \vert T_1} \times \cdots \times a_{\ell \vert T_k} \right) \\
        &= a_{2} \otimes \left( a_1 \times a_{34}\right) .
    \end{align*}
\end{example}

\subsubsection{The Hopf Monoid of Orbit Polytopes}\label{sec:orbitpolytopes}
We now consider the submonoid of $\nGP$ which results from restricting to be orbit polytopes. 

In order to define our vector species, we must first describe normal equivalence classes of orbit polytopes, given by the second author \cite{Mariel}.
Every orbit polytopes has an associated composition of a natural number, which we explain next.
A \Def{composition} of $n$ is a way of writing $n$ as an ordered sum of positive integers. 
For example, the distinct compositions of $n=3$ are $1+1+1$, $1+2$, $2+1$, and $3$ which we denote as $(1,1,1)$, $(1,2)$, $(2,1)$, and $(3)$ respectively.
In genenral, we denote compositions as $\lambda = (\lambda_1,\ldots,\lambda_k)$, meaning $\lambda_i > 0$ for each $i$ and $\sum_{i=1}^k \lambda_i = n$. We denote by $|\lambda|$ the sum of all the parts of $\lambda$.
Let $x\in\mathbb{R}^I$ and write the coordinates in decreasing order.
Construct a composition $\lambda$ where $\lambda_i$ be the number of times the $i$th largest coordinate appears in $x$.
We call such a $\lambda$ the \Def{composition of the point} $x$.
There is in fact a bijection between normal equivalence classes of orbit polytopes in $\mathbb{R}^I$, and compositions of $n = |I|$; two orbit polytopes are normally equivalent if and only if their vertices have the same composition \cite{Mariel}.
Given a composition $\lambda$ with $k$ parts, the associated \Def{orbit polytope normal equivalence class of} $\lambda$, denoted $\Or_\lambda$, is the normal equivalence class of $\Or(x)$ where $x$ is a point defined as
\begin{equation} \label{pcomp}
    x = (\underbrace{k,\ldots,k}_{\lambda_1 \text{ entries}}, \ldots,\underbrace{2,\ldots,2}_{\lambda_{k-1} \text{ entries}}, \underbrace{1,\ldots,1}_{\lambda_k \text{ entries}}).
\end{equation}

Define the \Def{Hopf monoid of orbit polytopes up to normal equivalence} $\nOP$ by the following.
Orbit polytopes are generalized permutahedra, and their normal equivalence classes $\Or_\lambda$ form a submonoid $\nOP$ of $\nGP$ \cite{Mariel}.
We define the vector space
\[
\nOP[I] = \mathrm{span}\{ \Or_{\alpha_1, S_1} \times \cdots \times \Or_{\alpha_k, S_k} : I = S_1 \sqcup \cdots \sqcup S_k \text{ and } \alpha_i \text{ composition of } |S_i| \text{ for all } i\}.
\]
where $\Or_{\alpha_i, S_i}$ is the orbit polytope normal equivalence class of $\alpha_i$ in the entries of the coordinates indexed by the set $S_i$.

To describe the coproduct of $\nOP$, we use two operations on compositions..
Given two compositions $\alpha = (\alpha_1, \ldots, \alpha_k)$ and $\beta = (\beta_1, \ldots, \beta_\ell)$ define 
\begin{itemize}
    \item the \Def{concatenation} $\alpha \cdot \beta$ as $(\alpha_1, \ldots, \alpha_k, \beta_1, \ldots, \beta_\ell)$;
    \item the \Def{near-concatenation} $\alpha \odot \beta$ as $(\alpha_1, \ldots, \alpha_{k-1} ,\alpha_k + \beta_1, \beta_2, \ldots, \beta_\ell)$.
\end{itemize}

\begin{theorem}[\cite{Mariel}]
Given a decomposition $I = S \sqcup T$ and a composition $\alpha$ of $n=|I|$, we can describe the coproduct of $\nOP$ as
\begin{equation}\label{eq:coproductorbit}
    \Delta_{S,T}(\Or_\alpha) = \Or_{\alpha \vert_S} \otimes \Or_{\alpha /_S}
\end{equation}
where $\alpha\vert_S$ and $\alpha/_S$ are the unique pair of compositions satisfying 
\begin{enumerate}
    \item $\alpha\vert_S$ is a composition of $|S|$ and $\alpha/_S$ is a composition of $|T|$, and
    \item either $\alpha\vert_S \cdot \alpha/_S = \alpha$ or $\alpha\vert_S \odot \alpha/_S = \alpha$.
\end{enumerate}
\end{theorem}

\section{Lie monoids of Primitives}\label{sec:lieprimitives}

We now apply the theory in Section~\ref{subsec:hopfmonoids} to the Hopf monoids in Section~\ref{sec:hopfmonoidexamples}.
In particular, when we dualize $\nGP$ and its Hopf submonoids, we obtain cocommutative Hopf monoids of particular interest.
Theorem~\ref{thrm:CMM} tells us that these are determined by their primitive elements.
This theorem motivates us to describe these primitive elements and their Lie monoid structure which we do here.

\begin{theorem}\label{thm:posetprimitives}
    The primitives of the dual Hopf monoid of posets, $\mathbf{P}^*$, are the duals of connected posets. 
\end{theorem}
\begin{proof}
    By Theorem~\ref{thrm:dualpi}, the primitives of $\mathbf{P}^*$ are the duals of indecomposables in $\mathbf{P}$.
    The product of two posets is their disjoint union, so the indecomposable elements of $\mathbf{P}$ are the connected posets. 
\end{proof}

In fact, the Hopf monoids $\nGP$, $\nPI$, $\nA$, and $\nOP$ all have a notion of ``connected'' elements that cannot be expressed as nontrivial products of other elements.
We use this to determine the primitives of $\nGP$, and the description of the primitives of its submonoids follows as a corollary.

\begin{theorem}\label{thm:GPprimitives}
    The primitives of $\nGP^*$ are the duals to the $(|I|-1)$-dimensional generalized permutahedra in $\R^I$. 
\end{theorem}
\begin{proof}
    Like with the proof of Theorem~\ref{thm:posetprimitives}, we find the primitives in $\nGP^*$ by taking the duals of the indecomposable elements of $\nGP$.
    These are generalized permutahedra in $\R^I$ that cannot be expressed as products of two or more generalized permutahedra, which are exactly those of codimension $1$.
\end{proof}

\begin{corollary}\label{cor:subGPprimitives}
    The primitives of $\nPI^*$, $\nA^*$, and $\nOP^*$ are as follows:
    \begin{enumerate}[label=(\roman*)]
        \item For $\nPI^*[I]$, these are $\pi_I^*$. 
        \item For $\nA^*[I]$, these are $a_\ell^*$ for any linear order $\ell$ on $I$. 
        \item For $\nOP^*[I]$, these are $\Or_{\alpha, I}^*$ where $\alpha$ is a composition of $|I|$.
    \end{enumerate}
\end{corollary}

Recall Theorem~\ref{thm:primitivesareliesub} shows that the primitives are Lie monoids with the the Lie bracket given by the commutator. 
We will next describe this Lie bracket for the the primitives of $\nPI^*$, $\nA^*$, and $\nOP^*$.
In order to give the Lie bracket we first compute the dual product by applying Proposition~\ref{prop:dualprodcoprod}. 

\subsection{Primitives of Dual Permutahedra}
We first consider computing the product and Lie bracket for $\mathcal{P}(\nPI^*)$.
\begin{theorem}\label{thm:permprims}
    Let $I = S \sqcup T$, then the product of $\pi^*_S \in \pp{\nPI^*}[S]$ and $\pi^*_T \in \pp{\nPI^*}[T]$ in $\pp{\nPI^*}$ is $\mu_{S,T}^*(\pi_S^* \otimes \pi_T^*) = \pi_{S \sqcup T}^* + (\pi_S \times \pi_T)^*$.
\end{theorem}   
\begin{proof}
    By Proposition~\ref{prop:dualprodcoprod}, we can define the dual product as 
    \begin{equation}\label{eq:dualprimperm}
         \mu_{S,T}^*(\pi_S^* \otimes \pi_T^*) = \sum_{\Delta_{S,T}(z) = \pi_S \otimes \pi_T} z^* 
     \end{equation}
     in $\nPI^*$.
     Geometrically speaking, this is a formal sum over the duals of all permutahedra and products of permutahedra whose $\1_S$-maximal face is the product $\pi_S \times \pi_T$. 
     The only possibilities are $\pi_{S \sqcup T}$ and $\pi_S \times \pi_T$, so $\mu^*_{S,T}(\pi_S^* \otimes \pi_T^*) = (\pi_S \times \pi_T)^* + \pi_{S \sqcup T}^*$.
\end{proof}

This leads us to the Lie bracket for $\mathcal{P}(\nPI^*)$.
In this case, the Lie bracket is identically zero because the product in $\mathcal{P}(\nPI^*)$ is commutative.
 
\begin{theorem}\label{thm:liebracketpermutahedra}
    The Lie monoid of primitives of $\nPI^*$ has Lie bracket given by the commutator $\gamma_{S,T}: \pp{\nPI^*}[S] \otimes \pp{\nPI^*}[T] \to \pp{\nPI^*}[I]$ where $\gamma_{S,T}(\pi_S^* \otimes \pi_T^*) = 0$.
\end{theorem}
\begin{proof}
     From Theorem~\ref{thm:permprims}, it follows that the commutator is
     \begin{align*}
         \gamma_{S,T}(\pi_S^* \otimes \pi_T^*) &= \mu^*_{S,T}(\pi_S^* \otimes \pi_T^*) - \mu^*_{T,S}(\pi_T^* \otimes \pi_S^*) \\
         &= (\pi_S \times \pi_T)^* + \pi_{S \sqcup T}^* - (\pi_T \times \pi_S)^* - \pi_{T \sqcup S}^* \\
         &= 0.
     \end{align*}
 \end{proof}

\subsection{Primitives of Dual Associahedra and the Witt Lie Algebra}

Next we find the commutator for $\pp{\nA^*}$.
In order to understand the product in $\pp{\nA^*}$, we will need a notion of block insertions of linear orders.
For disjoint sets $S$ and $T$, let $\ell$ be a linear order $s_1<_\ell s_2<_\ell\cdots<_\ell s_{|S|}$ on a finite set $S$ and $m$ be a linear order $t_1<_m t_2 <_m\cdots <_m t_{|T|}$ on a finite set $T$.
We say $b$ is a \Def{block insertion} of $m$ into $\ell$ if there exists some index $i$ with $0\le i\le |S|$ such that $b$ is the linear order $s_1<_b \cdots <_b s_i <_b t_1<_b\cdots <_b t_{|T|}<_b s_{i+1}<_b\cdots<_b s_{|S|}$.
In other words, $m$ is inserted into $\ell$ at index $i$ (where $i=0$ means $t_1$ is the least element in the order and $t_{|T|}<s_1$).
Note that there are $|S|+1$ ways to block insert $m$ into $\ell$.

\begin{figure}[H]
     \centering
     \begin{tikzpicture}[nod/.style={draw=black,fill=black,circle,inner sep=1.5pt}, cover/.style={ thick, norange}, scale=0.45]
    \node[above,bblue] at (3,0) {$1$}; 
    \node[above,bblue] at (4,0) {$4$};
    \node[above,bblue] at (5,0) {$2$};
    \node at (11.1,0) {linear order $m$ on $\{1,4,2\}$};
    
    \node[above,npink] at (3.5,-2) {$3$}; 
    \node[above,npink] at (4.5,-2) {$5$};
    \node at (10.5,-2) {linear order $\ell$ on $\{3,5\}$};
    
    \draw[bblue] (3,0) node[nod,bblue] {} -- (4,0) node[nod,bblue] {} -- (5,0) node[nod,bblue] {};
    
    \draw[npink] (3.5,-2) node[nod,npink] {} -- (4.5,-2) node[nod,npink] {};
    
    \node at (9.5,-4) {All block insertions of $m$ into $\ell$:};
    
    \draw[bblue] (1.5,-6) node[nod,bblue]{} -- (2.5,-6) node[nod,bblue]{} -- (3.5,-6) ;
    \draw (3.5,-6)node[nod,bblue]{}-- (4.5,-6)node[nod,npink]{};
    \draw[npink] (4.5,-6) node[nod,npink]{} -- (5.5,-6) node[nod,npink]{};
    
    \node[above,bblue] at (1.5,-6) {$1$};
    \node[above,bblue] at (2.5,-6) {$4$};
    \node[above,bblue] at (3.5,-6) {$2$};
    \node[above,npink] at (4.5,-6) {$3$};
    \node[above,npink] at (5.5,-6) {$5$};
    
    \draw (7.5,-6) node[nod,npink]{} -- (8.5,-6) node[nod,bblue]{};
    \draw[bblue] (8.5,-6) -- (9.5,-6) node[nod,bblue]{} -- (10.5,-6);
    \draw (10.5,-6) node[nod,bblue]{} -- (11.5,-6) node[nod,npink]{};
    
    \node[above,npink] at (7.5,-6) {$3$};
    \node[above,bblue] at (8.5,-6) {$1$};
    \node[above,bblue] at (9.5,-6) {$4$};
    \node[above,bblue] at (10.5,-6) {$2$};
    \node[above,npink] at (11.5,-6) {$5$};

    \draw[npink] (13.5,-6) node[nod,npink]{} -- (14.5,-6);
    \draw (14.5,-6)node[nod,npink]{} -- (15.5,-6);
    \draw[bblue] (15.5,-6) node[nod,bblue]{} -- (16.5,-6) node[nod,bblue]{} -- (17.5,-6) node[nod,bblue]{};
    
    \node[above,npink] at (13.5,-6) {$3$};
    \node[above,npink] at (14.5,-6) {$5$};
    \node[above,bblue] at (15.5,-6) {$1$};
    \node[above,bblue] at (16.5,-6) {$4$};
    \node[above,bblue] at (17.5,-6) {$2$};
    
\end{tikzpicture}
     \caption{An example of block insertions}
     \label{fig:blockinsertion}
 \end{figure}

\begin{theorem}\label{thrm:DualProdAssociahedra}
    Let $I = S \sqcup T$; then the product of two dual associahedra $a^*_\ell \in \pp{\nA^*}[S]$ and $a^*_m \in \pp{\nA^*}[T]$ in $\pp{\nA^*}$ is 
    \[
    \mu^*_{S,T}(a^*_\ell \otimes a^*_m) = \sum_{\substack{b \text{ a block insertion } \\ \text{of } m \text{ into } \ell}} a^*_b + (a_\ell \times a_m)^*.
    \]
    
\end{theorem}

\begin{proof}
     Let $I = S \sqcup T$ be some decomposition of a finite set $I$.
     By Proposition~\ref{prop:dualprodcoprod}, the product in the dual Hopf monoid $\nA^*$ is
     \begin{equation}\label{eq:dualprimas}
         \mu_{S,T}^*(a_\ell^* \otimes a_m^*) = \sum_{\Delta_{S,T}(z) = a_\ell \otimes a_m} z^* 
     \end{equation}
     where $\ell$ is a linear order on $S$ and $m$ is a linear order on $T$.
     
     We will first consider when $z$ is not a product of associahedra, meaning $z = a_b$ for some linear order $b$ on $I$.
     We have that $\Delta_{S,T}(a_b) = a_{b \vert S} \otimes \left(a_{b \vert T_1} \times \cdots \times a_{b \vert T_k} \right)$ where $T = T_1 \sqcup \cdots \sqcup T_k$ is the decomposition of $T$ into maximal intervals of $b$ from \eqref{eq:coproductas}.
     Suppose that 
     \[
     a_{b \vert S} \otimes \left(a_{b \vert T_1} \times \cdots \times a_{b \vert T_k} \right) = a_\ell \otimes a_m.
     \]
     This means that $a_{b\vert S} = a_\ell$ and $a_{b \vert T_1} \times \cdots \times a_{b \vert T_k} = a_m$, implying that the product in the latter equation consists of only one term, since an associahedron can never be equal to the product of two or more associahedra.
     So $b \vert T$ should have a single maximal interval, which means that $b$ must be a block insertion of $m$ into $\ell$.
    
     We will now show the only product of associahedra whose coproduct in $\nA$ results in  $a_\ell \otimes a_m$ must be $a_\ell \times a_m$. Suppose we have a product of two arbitrary associahedra $a_o \times a_p$ with $o$ a linear order on $S$ and $p$ a linear order on $T$.
     Then the compatibility property of the product and the coproduct of $\nA$ results in the following:
     \begin{align*}
         \Delta_{S,T}(a_o \times a_p) &= \Delta_{S,T}(\mu_{S,T}(a_o, a_p)) \\
         &= (\mu_{S,\emptyset} \otimes \mu_{\emptyset,T})(\id_S \otimes \beta_{\emptyset,\emptyset} \otimes \id_T)(\Delta_{S,\emptyset}(a_o) \otimes \Delta_{\emptyset,T}(a_p)) \\
         &= (\mu_{S,\emptyset} \otimes \mu_{\emptyset,T})(a_o \otimes 1 \otimes 1 \otimes a_p)\\
         &= a_o \otimes a_p.
     \end{align*}
     So $ \Delta_{S,T}(a_o \times a_p) = a_\ell \otimes a_m$ can only happen when $a_o = a_\ell$ and $a_p = a_m$.
     
     We do not get any terms where $z$ is a product of more than two associahedra because $a_\ell$ and $a_m$ are not products of associahedra.
\end{proof}


\begin{theorem}\label{thm:liebracketassociahedra}
    Let $\ell$ and $m$ be linear orders on $S$ and $T$ respectively.
    The Lie bracket on $\mathcal{P}(\nA^*)$ is given by the commutator $\gamma_{S,T}: \mathcal{P}(\nA^*)[S] \otimes \mathcal{P}(\nA^*)[T] \to \mathcal{P}(\nA^*)[I]$ which has the form
    \[
    \gamma_{S,T}(a_\ell^* \otimes a_m^*) = \sum_{\substack{b \text{ a block insertion } \\ \text{of } m \text{ into } \ell}} a^*_b - \sum_{\substack{c \text{ a block insertion } \\ \text{of } \ell \text{ into } m}} a^*_c.\]
\end{theorem}

Moreover, we can consider the Hopf algebra $\overline{A}^*:=\mathcal{F}(\nA^*)$ obtained by applying the Fock functor to $\nA^*$ mentioned in Section~\ref{subsec:hopfmonoids}.
Since $\nA$ is a finite dimension Hopf monoid, so is $\nA^*$, hence $\mathcal{F}(\mathcal{P}(\nA^*)) = \mathcal{P}(\mathcal{F}(\nA^*)) = \mathcal{P}(\overline{A}^*)$ as described in Section~\ref{subsec:CCM}.
The basis elements of $A^*$ have the form $a_n^*$ where $a_n:=a_{[n]}$ is the isomorphism class of all associahedra $a_I$ where $|I|=n$.

\begin{theorem}\label{thrm:DualAssociahedraLieAlg}
    The Lie algebra of primitives $\pp{\overline{A}^*}$ is generated by elements $a_1^*, a_2^*, \ldots$ with the Lie bracket $[a^*_s, a^*_t] = (s-t) a^*_{s+t}$.
\end{theorem}
 \begin{proof}
     Let $s,t \in \mathbb{N}$ and consider the dual associahedra $a^*_s, a^*_t \in A^*$. 
     By taking isomorphism classes we can see that the product from Theorem~\ref{thrm:DualProdAssociahedra} extends to the multiplication $a^*_s \cdot a^*_t = (s+1) a_{s+t}^* + (a_s \times a_t)^*$ in $A^*$.
     Likewise $a^*_t \cdot a^*_s = (t+1) a_{t+s}^* + (a_t \times a_s)^*$. So the Lie bracket is
     \begin{align*}
         [a^*_s, a^*_t] &= a^*_s \cdot a^*_t - a^*_t \cdot a^*_s \\
         &= (s+1) a_{s+t}^* + (a_s \times a_t)^* - \left((t+1) a_{t+s}^* + (a_t \times a_s)^* \right) \\
         &= s a_{s+t}^* + a_{s+t}^* + (a_s \times a_t)^* - t a_{t+s}^* - a_{t+s}^* - (a_t \times a_s)^* \\
         &= (s-t) a_{s+t}^*.
     \end{align*}
     This follows from the fact that $a_s\times a_t$ and $a_t\times a_s$ are isomorphic. 
 \end{proof}

The Lie algebra above looks similar to the Witt algebra, which is defined as follows \cite{Schottenloher2008}.
Take $\mathbb{C}[z,z^{-1}]$ to be the ring of Laurent polynomials with variable $z$. A \Def{derivation} is a linear map $D: \mathbb{C}[z,z^{-1}] \to \mathbb{C}[z,z^{-1}]$ that satisfies the rule $D(fg) = D(f) g + f D(g)$.
The \Def{Witt algebra} $W$ is the space of derivations of $\mathbb{C}[z,z^{-1}]$.
It has a basis given by $L_n = - z^{n+1} \frac{d}{dz}$ for $n \in \mathbb{Z}$ and is a Lie algebra with $[L_m,L_n] = (m-n) L_{m+n}$ for $m,n \in \mathbb{Z}$.



\begin{theorem}\label{thm:liealgebraassociahedra}
    The Lie algebra of primitives $\pp{\overline{A}^*}$ is isomorphic to the positive part of the Witt Lie algebra.
\end{theorem}

The previous theorem and Theorem~\ref{thrm:CMM} means that the universal enveloping algebra of the positive Witt Lie algebra is isomorphic to $\overline{A}^*$. 
Recent work has shown that the universal enveloping algebra of the positive Witt Lie algebra, i.e. $\overline{A}^*$, is not Noetherian using algebraic geometry \cite{wittnoetherian}.
This gives us a question on if this can be proven, or other algebraic properties can be deduced, using the structure of $\overline{A}^*$ combinatorially.

\subsection{Primitives of Dual Orbit Polytopes}
We next consider $\pp{\nOP^*}$ and the commutator $\gamma$ in this context.

\begin{theorem}
    The product of the dual Hopf monoid of orbit polytopes is
    \[
    \mu^*_{S,T}(\Or^*_\alpha \otimes \Or^*_\beta) = \begin{cases}
        \Or^*_{\alpha \cdot \beta} + \Or^*_{\alpha \odot \beta} + (\Or_\alpha \times \Or_\beta)^*  & \text{if } \alpha \text{ or } \beta \text{ have more than one part,} \\ 
        \Or^*_{\alpha \cdot \beta} + \Or^*_{\alpha \odot \beta} & \text{otherwise,}
    \end{cases}
    \]
    for $\alpha$ a composition of $|S|$ and $\beta$ a composition of $|T|$, so that $\Or_\alpha \in \pp{\nOP}[S], \Or_\beta \in \pp{\nOP}[T]$. 
\end{theorem}
 \begin{proof}
     Let $I = S \sqcup T$ be some decomposition of a finite set $I$.
     By Proposition~\ref{prop:dualprodcoprod}, the product in the dual Hopf monoid $\nOP^*$ is
     \begin{equation*}
         \mu^*_{S,T}(\Or^*_\alpha \otimes \Or^*_\beta) = \sum_{\Delta_{S,T}(z)= \Or_\alpha \otimes \Or_\beta} z^*.
    \end{equation*}
    
     We can first consider the case when $z \in \nOP[I]$ is a single orbit polytope. So say $z = \Or_\lambda$ for some composition $\lambda$ of $|I|$. This means $\Delta_{S,T}(\Or_\lambda) = \Or_\alpha \otimes \Or_\beta$ which holds if and only if $\Or_{\lambda \vert_S} \otimes \Or_{\lambda /_S} = \Or_\alpha \otimes \Or_\beta$ where $\lambda\vert_S$ and $\lambda/_S$ are the unique pair of compositions satisfying 
    \begin{enumerate}
         \item $\lambda\vert_S$ is a composition of $|S|$ and $\lambda/_S$ is a composition of $|T|$, and
         \item either $\lambda\vert_S \cdot \lambda/_S = \lambda$ or $\lambda\vert_S \odot \lambda/_S = \lambda$.
     \end{enumerate}
     The first condition means that our compositions must be the same, $\lambda\vert_S = \alpha$ and $\lambda/_S = \beta$. The second condition tells us that the only compositions which $\lambda$ can be are $\alpha \cdot \beta$ or $\alpha \odot \beta$.
     So we have the terms $\Or^*_{\alpha \cdot \beta}$ and $\Or^*_{\alpha \odot \beta}$ appearing in the expression for $\mu_{S,T}^*(\Or_\alpha^* \otimes \Or_\beta^*)$.
    
     Now suppose $z \in \nOP[I]$ is a product of two or more orbit polytopes, i.e. $I =S_1 \sqcup \cdots \sqcup S_k$ and $z = \Or_{\lambda_1} \times \cdots \times \Or_{\lambda_n}$ for $\lambda_i$ a composition of $|S_i|$. 
    
     If $\alpha$ and $\beta$ have more than one part, then $\dim \Or_\alpha = |S|-1$ and $\dim \Or_\beta = |T|-1$.
     Moreover $\dim z \leq |I|-k$ since the dimension of $\Or_{\lambda_i}$ is less than or equal to $|S_i|-1$.
     In order for the inequality 
     \[
     \Delta_{S,T}(\Or_{\lambda_1,S_1} \times \cdots \times \Or_{\lambda_n,S_k}) = \Or_{\alpha} \otimes \Or_{\beta}
     \]
     to hold, we need $\dim z\ge \dim\Or_\alpha +\dim\Or_\beta$, meaning $k\le 2$.
     We have already explored the case $k=1$.
     For $k=2$, this means $z= \Or_{\lambda_1} \times \Or_{\lambda_2}$.
     The compatibility axiom gives 
     \begin{align*}
         \Delta_{S,T}(\Or_{\lambda_1} \times \Or_{\lambda_2}) &= \Delta_{S,T}(\mu_{S,T}(\Or_{\lambda_1} \otimes \Or_{\lambda_2})) \\
         &= (\mu_{S,\emptyset} \otimes \mu_{\emptyset,T})(\id_S \otimes \beta_{\emptyset,\emptyset} \otimes \id_T)(\Delta_{S,\emptyset}(\Or_{\lambda_1}) \otimes \Delta_{\emptyset,T}(\Or_{\lambda_2})) \\
         &= (\mu_{S,\emptyset} \otimes \mu_{\emptyset,T})(\Or_{\lambda_1} \otimes 1 \otimes 1 \otimes \Or_{\lambda_2})\\
         &= \Or_{\lambda_1} \otimes \Or_{\lambda_2}.
     \end{align*}
    Thus $\alpha = \lambda_1$ and $\beta = \lambda_2$.
    This gives rise to the term $(\Or_\alpha \times \Or_\beta)^*$ in the expression for $\mu_{S,T}^*(\Or_\alpha \otimes \Or_\beta^*)$
    Since $\alpha$ and $\beta$ have more than one part, the terms $\Or_{\alpha \cdot \beta}$ and $\Or_{\alpha \odot \beta}$ cannot be expressed as products of orbit polytopes, so we have found a third distinct term.
    
    If either of $\alpha$ or $\beta$ have one part, then the corresponding orbit polytope is just a point.
    In this case, $(\Or_\alpha \times \Or_\beta)^*=\Or_{\alpha\odot\beta}^*$, so $\mu_{S,T}$ consists of only two terms.
     
 \end{proof}

With this product we can now find the lie bracket on the primitives $\pp{\nOP^*}$.

\begin{theorem}\label{thm:liebracketorbit}
    The commutator $\gamma: \mathcal{P}(\nOP^*)[S] \otimes \mathcal{P}(\nOP^*)[T] \to \mathcal{P}(\nOP^*)[I]$ is
    \[
    \gamma(\Or^*_\alpha \otimes \Or^*_\beta) =  \Or^*_{\alpha \cdot \beta}  - \Or^*_{\beta \cdot \alpha} + \Or^*_{\alpha \odot \beta} - \Or^*_{\beta \odot \alpha}. 
    \]
\end{theorem}

\section{The Brion Map on Generalized Permutahedra} \label{sec:brionmap}

So far we described the internal structure of our Hopf monoids in Section~\ref{sec:lieprimitives}.
We now turn to the external perspective of studying a Hopf monoid by exploring Hopf monoid morphisms, in particular the Brion map.
The Brion map was introduced by Aguiar and Ardila \cite{AguiarArdila} and described explicitly by Ardila and Sanchez \cite{ArdilaSanchez}.
Motivated by Brion's theorem \cite{Brion1988}, the Brion map connects a generalized permutahedron with posets coming from the vertices, extending the correspondence described in Section~\ref{subsec:gp}.

\begin{definition}
The \Def{Brion map} $B:\nGP[I] \to \P[I]$ is defined as
$$
B(P) = \sum_{v \text{ vertex of } P } \poset_v(P)
$$
where $\poset_v(P)$ is the poset corresponding to the tangent cone of the generalized permutahedron $P$ at the vertex $v$.
\end{definition}

It was shown that this is a morphism of Hopf monoids \cite{ArdilaSanchez}.
In this section, we give explicit expressions for the Brion map when the domain is restricted to the Hopf monoids of $\nPI$, $\nA$, and $\nOP$.

\subsection{The Brion Map on Permutahedra} \label{subsec:brionmappermutahedra}

We will first consider the Hopf monoid of permutahedra $\nPI$.
A \Def{chain} $c$ over the set $I$ is a poset where any two elements are comparable; for any $x \in I$ and any $y \in I$ we have $x \leq_c y$ or $y \leq_c x$. 
\begin{theorem}\label{thm:brionmappermutahedra}
    The Brion map on the Hopf monoid of permutahedra is
        $$B(\pi_I) = \sum_{c \text{ chain over } I} c.$$
\end{theorem}

\begin{proof}
     Suppose $|I| = n$ and consider a vertex $v$ of the permutahedron, which has the general form 
     \[
         v = (n,n-1,\ldots,1) = n e_{i_1} + \cdots + (n-(j-1)) e_{i_j} + (n-j) e_{i_j +1} + \cdots + 1 e_{i_n}
     \]
     where $I = \{i_1,\ldots,i_n \}$.
     The vertices of $\pi_I$ neighboring $v$ are those obtained by swapping the coefficient of $e_{i_j}$ and $e_{i_{j+1}}$ in this expression for some $1\le j\le n-1$.
     Thus the vector obtained by subtracting $v$ from one of its neighboring vertices $w$ is $w-v=e_{i_{j+1}}-e_{i_j}$.
     This corresponds to the relation $i_{j+1} \geq i_{j}$ in $\poset_v(\pi_I)$.
     Thus $\poset_v(\pi_I)$ is simply the chain $i_n \geq i_{n-1} \geq \ldots \geq i_{2} \geq i_1 $ over the set $I$.
     
     Since the vertices of $\pi_I$ are exactly the orbit of $v$ under the action of the symmetric group $S_I$ on $\mathbb{R}I$ by permuting coordinates, each chain over $I$ appears as a poset corresponding to a vertex of $\pi_I$ exactly once.
\end{proof}

We can also consider the Brion map on the Hopf algebra of standard permutahedra $\Pi$ by applying the Fock functor.
All chains on $n$ elements belong to the equivalence class $c_n$ of unlabelled chains.
So after applying the Fock functor to $\nPI$ to pass to the Hopf algebra $\Pi$ of standard permutahedra, the Brion map becomes 
\[
    B(\pi_n)=n!\, c_n.
\]

\subsection{The Brion Map on Associahedra}\label{subsec:brionmapassociahedra}

We define a \Def{rooted binary tree} recursively.
We say that the empty set is a binary tree. Otherwise, a binary tree has a root vertex $v$, a left sub-tree $T_1$, and a right sub-tree $T_2$ which are both rooted binary trees. 
We will follow the convention of Loday \cite{Loday,Loday2005} which depicts a rooted binary tree with the root at the bottom.

\begin{figure}
        \centering
        \scalebox{0.7}{
        \begin{tikzpicture}[nod/.style={draw=black,fill=black,circle,inner sep=1.5pt}, edge/.style={color=norange, line width=2}, scale=0.40]

\node at (-2.35,0) {$ $};

    \draw[black,thick] (1,0) -- (-2,3);
    \draw[black,thick] (-1,2) -- (0,3);
    \draw[black,thick] (0,1) -- (2,3);
    \draw[black,thick] (1,0) -- (4,3);
    
    \draw[black,thick] (1,-1) -- (1,0);
    
    \draw[edge] (1,0) -- (-1,2);

    \draw[black,thick] (8,0) -- (5,3);
    \draw[black,thick] (8,2) -- (7,3);
    \draw[black,thick] (7,1) -- (9,3);
    \draw[black,thick] (8,0) -- (11,3);
    
    \draw[black,thick] (8,-1) -- (8,0);
    
    \draw[edge] (8,0) -- (7,1) -- (8,2);

    \draw[black,thick] (15,0) -- (12,3);
    \draw[black,thick] (13,2) -- (14,3);
    \draw[black,thick] (17,2) -- (16,3);
    \draw[black,thick] (15,0) -- (18,3);
    
    \draw[black,thick] (15,-1) -- (15,0);
    
    \draw[edge] (15,0) -- (13,2);
    \draw[edge] (15,0)-- (17,2);

    \draw[black,thick] (22,0) -- (19,3);
    \draw[black,thick] (23,1) -- (21,3);
    \draw[black,thick] (22,2) -- (23,3);
    \draw[black,thick] (22,0) -- (25,3);
    
    \draw[black,thick] (22,-1) -- (22,0);
    
    \draw[edge] (22,0) -- (23,1) -- (22,2);

    \draw[black,thick] (29,0) -- (26,3);
    \draw[black,thick] (30,1) -- (28,3);
    \draw[black,thick] (31,2) -- (30,3);
    \draw[black,thick] (29,0) -- (32,3);
    
    \draw[black,thick] (29,-1) -- (29,0);
    
    \draw[edge] (29,0) -- (30,1) -- (31,2);

\end{tikzpicture}
        }
    \end{figure}
    \begin{figure}
        \centering
        \scalebox{0.7}{
        \begin{tikzpicture}[nod/.style={draw=norange,fill=norange,circle,inner sep=1.5pt}, cover/.style={ thick, norange}, scale=0.60]

    \coordinate[nod] (a) at (-2,2);
    \coordinate[nod] (b) at (-1,1);
    \coordinate[nod] (c) at (0,0);
    
    \draw[cover] (c) -- (b) -- (a);
    
    \node[anchor=south] at (b) {\color{black}};
    \node[anchor=east] at (a) {\color{black}};
    \node[anchor=west] at (c) {\color{black}};
    

    \coordinate[nod] (d) at (4,2);
    \coordinate[nod] (e) at (3,1);
    \coordinate[nod] (f) at (4,0);
    
    \draw[cover] (f) -- (e) -- (d);

    \coordinate[nod] (g) at (7,1);
    \coordinate[nod] (h) at (9,1);
    \coordinate[nod] (i) at (8,0);
    
    \draw[cover] (i) -- (h);
    \draw[cover] (i) -- (g);

    \coordinate[nod] (j) at (12,2);
    \coordinate[nod] (k) at (13,1);
    \coordinate[nod] (l) at (12,0);
    
    \draw[cover] (l) -- (k) -- (j);

    \coordinate[nod] (m) at (18,2);
    \coordinate[nod] (n) at (17,1);
    \coordinate[nod] (o) at (16,0);
    
    \draw[cover] (m) -- (n) -- (o);

\end{tikzpicture}
        }
    \caption{The rooted binary trees with $3$ internal vertices represented in two ways.}
    \label{figure:Y5RBTExample}
\end{figure}
Figure~\ref{figure:Y5RBTExample} shows two visual representations of a rooted binary tree; we will make use of both.
In the upper representation, each vertex is connected by an edge to both its left and right sub-trees, regardless of whether or not a sub-tree is empty.
Edges connecting to empty sub-trees are called ``leaves'' and this depiction is the ``leaves representation''.
The lower representation is obtained by deleting the leaves, leaving behind only the internal vertices and edges of the rooted binary tree.
We call this the ``pruned representation''.
If a rooted binary tree has $n$ internal vertices, then it has $n+1$ leaves.
Let $Y_n$ denote the set of all rooted binary trees with $n$ internal vertices.


Rooted binary trees are in correspondence with the vertices of associahedra. 
The following construction is by Loday \cite{Loday}.
For a given linear order $\ell=\ell_1<\ell_2<\dots<\ell_n$, introduce a new symbolic minimal element $\ell_0$.
Label the leaves of $T$ with $\ell_0,\ell_1\ldots,\ell_n$ going from left to right.
This induces a canonical ordering of the $n$ internal vertices of $T$; internal vertex $\ell_i$ is the internal vertex between leaves $\ell_{i-1}$ and $\ell_i$.
This is known as the \Def{binary search labelling} with respect to the linear order $\ell$ \cite[Section 8.1]{Postnikov}.

To each tree $T \in Y_n$, we can associate an integral point $v_\ell(T) \in \mathbb{R}^n$ as follows.
Define $l_i$ and $r_i$ to be the number of leaves that are descendants on the left side and right side of internal vertex $\ell_i$ in $T$, respectively.
We can now construct the point $v_\ell(T) = (l_1 r_1, \ldots, l_i r_i, \ldots, l_n r_n)$.
As an example, using the tree from Figure~\ref{fig:RBTlabeling} we have $v(T) = (2,1,6,1)$ corresponding to the standard ordering $1<2<3<4$.
The \Def{associahedron} $a_\ell$ is the convex hull of $\{v_\ell(T): T \in Y_n\}$ \cite{Loday, Postnikov}.

\begin{figure}
   \centering\scalebox{.7}{
   \begin{tikzpicture}[scale=0.450,baseline=-0.2ex]
    \coordinate (1) at (-2,2);
    \coordinate (2) at (-1,3);
    \coordinate (3) at (0,0);
    \coordinate (4) at (3,3);
    
    \node[left] at (1) {};
    \node[right] at (2) {};
    \node[left] at (3) {};
    \node[right] at (4) {};
    
    \draw (3) -- (1);
    \draw (1) -- (2);
    \draw (3) -- (4);
    
    \draw (2) -- (0,4);
    \draw (2) -- (-2,4);
    
    \draw (4) -- (2,4);
    \draw (4) -- (4,4);
    
    \draw (1) -- (-4,4);
    
    \draw (3) -- (0,-0.5);

\end{tikzpicture}\begin{tikzpicture}[scale=0.450,baseline=-0.2ex]
    \coordinate (1) at (-2,2);
    \coordinate (2) at (-1,3);
    \coordinate (3) at (0,0);
    \coordinate (4) at (3,3);
    
    \draw[->] (-5.75,2)--(-4.5,2);
    \node[draw,opacity=0] at (-5.5,2) {};
    \node[left] at (1) {};
    \node[right] at (2) {};
    \node[left] at (3) {};
    \node[right] at (4) {};
    
    \draw (3) -- (1);
    \draw (1) -- (2);
    \draw (3) -- (4);
    
    \draw (2) -- (0,4);
    \draw (2) -- (-2,4);
    
    \draw (4) -- (2,4);
    \draw (4) -- (4,4);
    
    \draw (1) -- (-4,4);
    
    \draw (3) -- (0,-0.5);
    
    \node[above] at (-4,4) {$0$};
    \node[above] at (-2,4) {$1$};
    \node[above] at (0,4) {$2$};
    \node[above] at (2,4) {$3$};
    \node[above] at (4,4) {$4$};

\end{tikzpicture}\begin{tikzpicture}[scale=0.450,baseline=-0.2ex]
    \coordinate (1) at (-2,2);
    \coordinate (2) at (-1,3);
    \coordinate (3) at (0,0);
    \coordinate (4) at (3,3);
    
    \draw[->] (-5.75,2)--(-4.5,2);
    \node[draw,opacity=0] at (-5.5,2) {};
    \node[left] at (1) {$1$};
    \node[right] at (2) {$2$};
    \node[left] at (3) {$3$};
    \node[right] at (4) {$4$};
    
    \draw (3) -- (1);
    \draw (1) -- (2);
    \draw (3) -- (4);
    
    \draw (2) -- (0,4);
    \draw (2) -- (-2,4);
    
    \draw (4) -- (2,4);
    \draw (4) -- (4,4);
    
    \draw (1) -- (-4,4);
    
    \draw (3) -- (0,-0.5);
    
    \node[above,opacity=0.35] at (-4,4) {$0$};
    \node[above,opacity=0.35] at (-2,4) {$1$};
    \node[above,opacity=0.35] at (0,4) {$2$};
    \node[above,opacity=0.35] at (2,4) {$3$};
    \node[above,opacity=0.35] at (4,4) {$4$};

\end{tikzpicture}}
   \caption{Labeling the internal vertices.}
   \label{fig:RBTlabeling}
\end{figure}

Moreover, $Y_n$ has a poset structure called the Tamari order \cite{Loday}. Given $A,B \in Y_n$, we have a covering relation $B \leq A$ if $A$ can be obtained from $B$ by replacing locally a sub-tree of $B$ of the form 
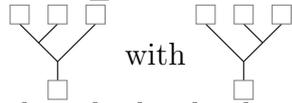
\begin{tikzpicture}[scale=0.50,baseline=-0.2ex]
    \draw (-1,1) -- (0,0);
    \draw (-0.5,0.5)--(0,1);
    \draw (1,1)-- (0,0);
    \draw (0,0) -- (0,-0.5);
    
    
    \draw[gray] (0.75,1) rectangle (1.25,1.5);
    \draw[gray] (-0.25,1) rectangle (0.25,1.5);
    \draw[gray] (-1.25,1) rectangle (-0.75,1.5);
    
    \draw[gray] (-0.25,-0.5) rectangle (0.25,-1);
\end{tikzpicture}
with 
\begin{tikzpicture}[scale=0.50,baseline=-0.2ex]
    \draw (-1,1) -- (0,0);
    \draw (0.5,0.5)--(0,1);
    \draw (1,1)-- (0,0);
    \draw (0,0) -- (0,-0.5);

    \draw[gray] (0.75,1) rectangle (1.25,1.5);
    \draw[gray] (-0.25,1) rectangle (0.25,1.5);
    \draw[gray] (-1.25,1) rectangle (-0.75,1.5);
    
    \draw[gray] (-0.25,-0.5) rectangle (0.25,-1);
\end{tikzpicture}.
This operation and its inverse are called \Def{tree rotations}.
Notice that the local sub-tree pictured here has exactly one internal edge, both before and after performing a tree rotation.
Hence each internal edge of a rooted binary tree $T$ corresponds to a unique tree rotation that can be performed on $T$.

Each covering relation in the Tamari order corresponds to an edge of the associahedron \cite{Loday2005}. 
Thus given the rooted binary tree associated to a vertex of the associahedron, the neighboring vertices are the ones associated to rooted binary trees obtained by performing a tree rotation.
From this we can determine the edge directions incident to each vertex of the associahedron.

In order to do this, we need to analyze how tree rotations affect the ordering of the internal vertices of the rooted binary tree.
Consider some internal edge of a tree $T$ with vertices $\ell_i$ and $\ell_j$, where $\ell_i$ is a child of $\ell_j$.
If $\ell_i$ is a left child of $\ell_j$, then the tree has the following form locally around $\ell_i$ and $\ell_j$, before and after performing a tree rotation:
        
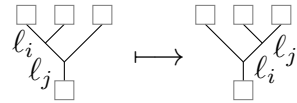
\begin{figure}[H]
            \centering
            \begin{tikzpicture}[scale=0.50,baseline=-0.2ex]
    \draw (-1,1) -- (0,0);
    \draw (-0.5,0.5)--(0,1);
    \draw (1,1)-- (0,0);
    \draw (0,0) -- (0,-0.5);

    \node[left] at (-0.5,0.5) {$\ell_i$};
    \node[left] at (0,-0.3) {$\ell_j$};
    
    \draw[gray] (0.75,1) rectangle (1.25,1.5);
    \draw[gray] (-0.25,1) rectangle (0.25,1.5);
    \draw[gray] (-1.25,1) rectangle (-0.75,1.5);
    
    \draw[gray] (-0.25,-0.5) rectangle (0.25,-1);
\end{tikzpicture}
            $\longmapsto$
            \begin{tikzpicture}[scale=0.50,baseline=-0.2ex]
    \draw (-1,1) -- (0,0);
    \draw (0.5,0.5)--(0,1);
    \draw (1,1)-- (0,0);
    \draw (0,0) -- (0,-0.5);
    
    \node[right] at (0.5,0.3) {$\ell_j$};
    \node[right] at (0,-0.2) {$\ell_i$};
    
    \draw[gray] (0.75,1) rectangle (1.25,1.5);
    \draw[gray] (-0.25,1) rectangle (0.25,1.5);
    \draw[gray] (-1.25,1) rectangle (-0.75,1.5);
    
    \draw[gray] (-0.25,-0.5) rectangle (0.25,-1);
\end{tikzpicture}
        \end{figure}
\noindent Similarly, if $\ell_i$ is a right child of $\ell_j$, then the tree before and after the rotation has the form:
    \begin{figure}[H]
        \centering
        \begin{tikzpicture}[scale=0.50,baseline=-0.2ex]
    \draw (-1,1) -- (0,0);
    \draw (0.5,0.5)--(0,1);
    \draw (1,1)-- (0,0);
    \draw (0,0) -- (0,-0.5);
    
    \node[right] at (0.5,0.3) {$\ell_i$};
    \node[right] at (0,-0.2) {$\ell_j$};
    
    \draw[gray] (0.75,1) rectangle (1.25,1.5);
    \draw[gray] (-0.25,1) rectangle (0.25,1.5);
    \draw[gray] (-1.25,1) rectangle (-0.75,1.5);
    
    \draw[gray] (-0.25,-0.5) rectangle (0.25,-1);
\end{tikzpicture}
        $\longmapsto$
        \begin{tikzpicture}[scale=0.50,baseline=-0.2ex]
    \draw (-1,1) -- (0,0);
    \draw (-0.5,0.5)--(0,1);
    \draw (1,1)-- (0,0);
    \draw (0,0) -- (0,-0.5);
    
    \node[left] at (-0.5,0.3) {$\ell_j$};
    \node[left] at (0,-0.2) {$\ell_i$};
    
    \draw[gray] (0.75,1) rectangle (1.25,1.5);
    \draw[gray] (-0.25,1) rectangle (0.25,1.5);
    \draw[gray] (-1.25,1) rectangle (-0.75,1.5);
    
    \draw[gray] (-0.25,-0.5) rectangle (0.25,-1);
\end{tikzpicture}
    \end{figure}
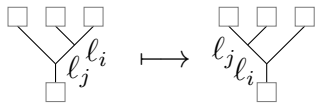
\noindent Figure~\ref{fig:edgebijection} shows how the labelling of the internal vertices changes from all possible rotations for a particular rooted binary tree.

\begin{figure}
    \centering
    
    
    
    
    
    


\begin{tikzpicture}[scale=0.40,baseline=-0.2ex]
    \coordinate (1) at (-2,2);
    \coordinate (2) at (-1,3);
    \coordinate (3) at (0,0);
    \coordinate (4) at (3,3);
    
    \node[left] at (1) {$1$};
    \node[right] at (2) {$2$};
    \node[left] at (3) {$3$};
    \node[right] at (4) {$4$};
    
    \draw[thick] (3) -- (1);
    
    \draw[line width=8pt, line cap = round, opacity=0.25, norange] (1) -- (2);
    \draw[line width=8pt, line cap = round, opacity=0.25, norange] (2) -- (0,4);
    \draw[norange,line width= 2pt] (1) -- (2);
    
    \draw[thick] (3) -- (4);
    
    \draw (2) -- (0,4);
    
    \draw[line width=8pt, line cap = round, opacity=0.25, norange] (2) -- (-2,4);
    \draw (2) -- (-2,4);
    
    \draw (4) -- (2,4);
    \draw (4) -- (4,4);
    
    \draw[line width=8pt, line cap = round, opacity=0.25, norange] (1) -- (-4,4);
    
    \draw (1) -- (-4,4);
    
    \draw (3) -- (0,-0.5);

\end{tikzpicture}\hspace{1em}\begin{tikzpicture}[scale=0.40,baseline=-0.2ex]
    \coordinate (1) at (-2,2);
    \coordinate (2) at (-1,3);
    \coordinate (3) at (0,0);
    \coordinate (4) at (3,3);
    
    \node[left] at (1) {$1$};
    \node[right] at (2) {$2$};
    \node[left] at (3) {$3$};
    \node[right] at (4) {$4$};
    
    \draw[line width=8pt, line cap = round, opacity=0.25, norange] (1) -- (2);
    \draw[thick] (1) -- (2);
    
    \draw[line width=8pt, line cap = round, opacity=0.25, norange] (3) -- (4);
    \draw[thick] (3) -- (4);
    
    \draw[line width=8pt, line cap = round, opacity=0.25, norange] (1) -- (-4,4);
    \draw (1) -- (-4,4);
    
    \draw[line width=8pt, line cap = round, opacity=0.25, norange] (3) -- (1);
    \draw[norange,line width=2pt] (3) -- (1);
    
    \draw (2) -- (0,4);
    \draw (2) -- (-2,4);
    
    \draw (4) -- (2,4);
    \draw (4) -- (4,4);

    \draw (3) -- (0,-0.5);

\end{tikzpicture}\hspace{1em}\begin{tikzpicture}[scale=0.40,baseline=-0.2ex]
    \coordinate (1) at (-2,2);
    \coordinate (2) at (-1,3);
    \coordinate (3) at (0,0);
    \coordinate (4) at (3,3);
    
    \node[left] at (1) {$1$};
    \node[right] at (2) {$2$};
    \node[left] at (3) {$3$};
    \node[right] at (4) {$4$};
    
    \draw[line width=8pt, line cap = round, opacity=0.25, norange] (3) -- (1);
    \draw[thick] (3) -- (1);
    
    \draw[thick] (1) -- (2);

    \draw (2) -- (0,4);
    \draw (2) -- (-2,4);
    
    \draw[line width=8pt, line cap = round, opacity=0.25, norange] (4) -- (2,4);
    \draw[line width=8pt, line cap = round, opacity=0.25, norange] (4) -- (4,4);
    \draw (4) -- (2,4);
    \draw (4) -- (4,4);
    
    \draw[line width=8pt, line cap = round, opacity=0.25, norange] (3) -- (4);
    \draw[norange,line width=2pt] (3) -- (4);
    
    \draw (1) -- (-4,4);
    
    \draw (3) -- (0,-0.5);

\end{tikzpicture}
\vspace{0.5em}

\begin{tikzpicture}[scale=0.40,baseline=-0.2ex]
    \coordinate (1) at (-2,2);
    \coordinate (2) at (-1,3);
    \coordinate (2') at (-3,3);
    \coordinate (3) at (0,0);
    \coordinate (4) at (3,3);
    
    \draw[->] (0,5.5) -- (0,4.5);
    
    \node[left] at (1) {$2$};
    \node[left] at (2') {$1$};
    \node[left] at (3) {$3$};
    \node[right] at (4) {$4$};
    
    \draw[thick] (3) -- (1);
    
    \draw[line width=8pt, line cap = round, opacity=0.25, norange] (1) -- (0,4);
    \draw[norange,line width= 2pt] (1) -- (2');
    
    \draw[thick] (3) -- (4);
    
    \draw (1) -- (0,4);
    
    \draw[line width=8pt, line cap = round, opacity=0.25, norange] (2') -- (-2,4);
    \draw (2') -- (-2,4);
    
    \draw (4) -- (2,4);
    \draw (4) -- (4,4);
    
    \draw[line width=8pt, line cap = round, opacity=0.25, norange] (2') -- (-4,4);
    \draw[line width=8pt, line cap = round, opacity=0.25, norange] (1) -- (2');
    \draw (2') -- (-4,4);
    
    \draw (3) -- (0,-0.5);

\end{tikzpicture}\hspace{1em}\begin{tikzpicture}[scale=0.40,baseline=-0.2ex]
    \coordinate (1) at (-2,2);
    \coordinate (1') at (1,1);
    \coordinate (2) at (-1,3);
    \coordinate (3) at (0,0);
    \coordinate (4) at (3,3);
    
    \draw[->] (0,5.5) -- (0,4.5);
    
    \node[right] at (1') {$3$};
    \node[right] at (2) {$2$};
    \node[left] at (3) {$1$};
    \node[right] at (4) {$4$};
    
    \draw[line width=8pt, line cap = round, opacity=0.25, norange] (1') -- (2);
    \draw[thick] (1') -- (2);
    
    \draw[line width=8pt, line cap = round, opacity=0.25, norange] (3) -- (1');
    \draw[line width=8pt, line cap = round, opacity=0.25, norange] (1') -- (4);
    
    \draw[norange,line width = 2pt] (3) -- (1');
    \draw[thick] (1') -- (4);
    
    \draw[line width=8pt, line cap = round, opacity=0.25, norange] (3) -- (-4,4);
    \draw (3) -- (-4,4);
    
    
    \draw (2) -- (0,4);
    \draw (2) -- (-2,4);
    
    \draw (4) -- (2,4);
    \draw (4) -- (4,4);

    \draw (3) -- (0,-0.5);

\end{tikzpicture}\hspace{1em}\begin{tikzpicture}[scale=0.40,baseline=-0.2ex]
    \coordinate (1) at (-2,2);
    \coordinate (2) at (-1,3);
    \coordinate (3) at (0,0);
    \coordinate (4) at (3,3);
    \coordinate (4') at (-1,1);
    
    \draw[->] (0,5.5) -- (0,4.5);
    
    \node[left] at (1) {$1$};
    \node[right] at (2) {$2$};
    \node[left] at (3) {$4$};
    \node[left] at (4') {$3$};

    \draw[line width=8pt, line cap = round, opacity=0.25, norange] (3) -- (4,4);

    \draw[line width=8pt, line cap = round, opacity=0.25, norange] (3) -- (4');
    \draw[line width=8pt, line cap = round, opacity=0.25, norange] (4') -- (1);
    
    \draw[norange,line width = 2pt] (3) -- (4');
    \draw[thick] (4') -- (1);
    
    \draw[thick] (1) -- (2);
    
    \draw[line width=8pt, line cap = round, opacity=0.25, norange] (4') -- (2,4);
    \draw (4') -- (2,4);

    \draw (2) -- (0,4);
    \draw (2) -- (-2,4);

    \draw (4) -- (4,4);
    
    \draw (3) -- (4);
    
    \draw (1) -- (-4,4);
    
    \draw (3) -- (0,-0.5);

\end{tikzpicture}
    \caption{Bijection between internal edges of a tree and rotations}
    \label{fig:edgebijection}
\end{figure}

\begin{figure}
        \centering
        \scalebox{0.7}{
        \begin{tikzpicture}[nod/.style={draw=black,fill=black,circle,inner sep=1.5pt}, edge/.style={color=norange, line width=2}, scale=0.40]

\node at (-2.35,0) {$ $};

    \draw[black,thick] (1,0) -- (-2,3);
    \draw[black,thick] (-1,2) -- (0,3);
    \draw[black,thick] (0,1) -- (2,3);
    \draw[black,thick] (1,0) -- (4,3);
    
    \draw[black,thick] (1,-1) -- (1,0);
    
    \draw[edge] (1,0) -- (-1,2);

    \draw[black,thick] (8,0) -- (5,3);
    \draw[black,thick] (8,2) -- (7,3);
    \draw[black,thick] (7,1) -- (9,3);
    \draw[black,thick] (8,0) -- (11,3);
    
    \draw[black,thick] (8,-1) -- (8,0);
    
    \draw[edge] (8,0) -- (7,1) -- (8,2);

    \draw[black,thick] (15,0) -- (12,3);
    \draw[black,thick] (13,2) -- (14,3);
    \draw[black,thick] (17,2) -- (16,3);
    \draw[black,thick] (15,0) -- (18,3);
    
    \draw[black,thick] (15,-1) -- (15,0);
    
    \draw[edge] (15,0) -- (13,2);
    \draw[edge] (15,0)-- (17,2);

    \draw[black,thick] (22,0) -- (19,3);
    \draw[black,thick] (23,1) -- (21,3);
    \draw[black,thick] (22,2) -- (23,3);
    \draw[black,thick] (22,0) -- (25,3);
    
    \draw[black,thick] (22,-1) -- (22,0);
    
    \draw[edge] (22,0) -- (23,1) -- (22,2);

    \draw[black,thick] (29,0) -- (26,3);
    \draw[black,thick] (30,1) -- (28,3);
    \draw[black,thick] (31,2) -- (30,3);
    \draw[black,thick] (29,0) -- (32,3);
    
    \draw[black,thick] (29,-1) -- (29,0);
    
    \draw[edge] (29,0) -- (30,1) -- (31,2);

\end{tikzpicture}
        }
        
        \scalebox{0.7}{
        \begin{tikzpicture}[nod/.style={draw=norange,fill=norange,circle,inner sep=1.5pt}, cover/.style={ thick, norange}, scale=0.60]

    \coordinate[nod] (a) at (-2,2);
    \coordinate[nod] (b) at (-1,1);
    \coordinate[nod] (c) at (0,0);
    
    \draw[cover] (c) -- (b) -- (a);
    
    \node[anchor=south] at (b) {\color{black}};
    \node[anchor=east] at (a) {\color{black}};
    \node[anchor=west] at (c) {\color{black}};

    \node[left] at (a) {$1$};
    \node[left] at (b) {$2$};
    \node[left] at (c) {$3$};
    

    \coordinate[nod] (d) at (4,2);
    \coordinate[nod] (e) at (3,1);
    \coordinate[nod] (f) at (4,0);
    
    \draw[cover] (f) -- (e) -- (d);

    \node[right] at (d) {$2$};
    \node[left] at (e) {$1$};
    \node[right] at (f) {$3$};

    \coordinate[nod] (g) at (7,1);
    \coordinate[nod] (h) at (9,1);
    \coordinate[nod] (i) at (8,0);
    
    \draw[cover] (i) -- (h);
    \draw[cover] (i) -- (g);

    \node[left] at (g) {$1$};
    \node[right] at (h) {$3$};
    \node[below] at (i) {$2$};

    \coordinate[nod] (j) at (12,2);
    \coordinate[nod] (k) at (13,1);
    \coordinate[nod] (l) at (12,0);
    
    \draw[cover] (l) -- (k) -- (j);

    \node[left] at (j) {$2$};
    \node[right] at (k) {$3$};
    \node[left] at (l) {$1$};

    \coordinate[nod] (m) at (18,2);
    \coordinate[nod] (n) at (17,1);
    \coordinate[nod] (o) at (16,0);
    
    \draw[cover] (m) -- (n) -- (o);

    \node[right] at (m) {$3$};
    \node[right] at (n) {$2$};
    \node[right] at (o) {$1$};

\end{tikzpicture}
        }
        
        \scalebox{0.7}{
        \begin{tikzpicture}[nod/.style={draw=norange,fill=norange,circle,inner sep=1.5pt}, cover/.style={ thick, norange}, scale=0.60]

    \coordinate[nod] (a) at (-1,2);
    \coordinate[nod] (b) at (-1,1);
    \coordinate[nod] (c) at (-1,0);
    
    \draw[cover] (c) -- (b) -- (a);
    
    \node[anchor=south] at (b) {\color{black}};
    \node[anchor=east] at (a) {\color{black}};
    \node[anchor=west] at (c) {\color{black}};

    \node[right] at (a) {$1$};
    \node[right] at (b) {$2$};
    \node[right] at (c) {$3$};
    

    \coordinate[nod] (d) at (3.5,2);
    \coordinate[nod] (e) at (3.5,1);
    \coordinate[nod] (f) at (3.5,0);
    
    \draw[cover] (f) -- (e) -- (d);

    \node[right] at (d) {$2$};
    \node[right] at (e) {$1$};
    \node[right] at (f) {$3$};

    \coordinate[nod] (g) at (7,1);
    \coordinate[nod] (h) at (9,1);
    \coordinate[nod] (i) at (8,0);
    
    \draw[cover] (i) -- (h);
    \draw[cover] (i) -- (g);

    \node[left] at (g) {$1$};
    \node[right] at (h) {$3$};
    \node[below] at (i) {$2$};

    \coordinate[nod] (j) at (12.5,2);
    \coordinate[nod] (k) at (12.5,1);
    \coordinate[nod] (l) at (12.5,0);
    
    \draw[cover] (l) -- (k) -- (j);

    \node[right] at (j) {$2$};
    \node[right] at (k) {$3$};
    \node[right] at (l) {$1$};

    \coordinate[nod] (m) at (17,2);
    \coordinate[nod] (n) at (17,1);
    \coordinate[nod] (o) at (17,0);
    
    \draw[cover] (m) -- (n) -- (o);

    \node[right] at (m) {$3$};
    \node[right] at (n) {$2$};
    \node[right] at (o) {$1$};

\end{tikzpicture}
        }
    \caption{From rooted binary trees, to labeling, to posets}
    \label{figure:Y5RBTexample2}
\end{figure}

Denote by $p_\ell(T)$ the poset whose diagram is obtained by taking the binary search labeling on the internal vertices of the rooted binary tree $T$, deleting the leaves to get the pruned representation, then forgetting the distinction between left and right sub-trees.
This associates a poset on $\{1,2,\dots,n\}$ to each $T\in Y_n$, that is to say, to each vertex of the associahedron.
For instance, the vertices of $a_4$ correspond to the posets in Figure~\ref{figure:Y5RBTexample2}.
Notice that these posets have a minimum element and are connected, with each element covered by at most two other elements.
They also depend on the linear order $\ell$ that is used.
We call posets that can be obtained in this way \Def{rooted binary tree (RBT) posets} and denote the set of all such posets $\operatorname{RBT}_\ell=\{p_\ell(T):T\in Y_n\}$ for the linear order $\ell$.

The following result can be proven utilizing work by Postnikov \cite[Proposition 7.10]{Postnikov}, but we provide a direct proof which uses the cover relations of the Tamari lattice.

\begin{theorem}\label{thm:brionmapassociahedra}
    The Brion map on the Hopf monoid of associahedra is given by
    \begin{equation}
        B(a_\ell) = \sum_{p\in\operatorname{RBT}_\ell} p.
        \label{eq:brionassoc}
    \end{equation}
\end{theorem}
 \begin{proof}
	We will show that the poset $\poset_{v_\ell(T)}(a_\ell)$ corresponding to the tangent cone of $a_\ell$ at the vertex $v_\ell(T)$ is the poset $p_\ell(T)$.
	Recall that the interior edges of $T$ are in bijection with tree rotations in $T$.
	So consider an internal edge of $T$, say between the vertices labelled $\ell_i$ and $\ell_j$ by the binary search labelling on $T$.
	Without loss of generality suppose $\ell_i$ is a child of $\ell_j$.
	Then there are two cases: $\ell_i$ is either a left child or a right child. 
         \begin{enumerate}[label=(\alph*)]
             \item Suppose $\ell_i$ is a left child of $\ell_j$, meaning the tree locally looks like
         
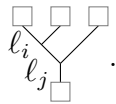
\[
             \begin{tikzpicture}[scale=0.50,baseline=-0.2ex]
    \draw (-1,1) -- (0,0);
    \draw (-0.5,0.5)--(0,1);
    \draw (1,1)-- (0,0);
    \draw (0,0) -- (0,-0.5);

    \node[left] at (-0.5,0.5) {$\ell_i$};
    \node[left] at (0,-0.3) {$\ell_j$};
    
    \draw[gray] (0.75,1) rectangle (1.25,1.5);
    \draw[gray] (-0.25,1) rectangle (0.25,1.5);
    \draw[gray] (-1.25,1) rectangle (-0.75,1.5);
    
    \draw[gray] (-0.25,-0.5) rectangle (0.25,-1);
\end{tikzpicture}.
         \]
	Recall that $l_k$ and $r_k$, the numbers of leaves to the left and right of node $\ell_k$, respectively, are used to define the vertex $v_\ell(T)$ of the associahedron.
	Notice that $l_j= l_i + r_i$ since $i$ is a left child of $j$.
         \begin{figure}
             \centering
             \begin{tikzpicture}[scale=0.250,baseline=-0.2ex]
    \node[] at (0,8.5) {$T$};
    
    \draw (-4,4) -- (0,0);
    \draw (-2,2)--(0,4);
    \draw (4,4)-- (0,0);
    \draw (0,0) -- (0,-2);

    \node[left] at (-2,1.4) {$\ell_i$};
    \node[left] at (0,-0.8) {$\ell_j$};
    
    \draw[gray] (3,4) rectangle (5,6);
    \node[] at (4,5) {$r_j$};
    
    \draw[gray] (-1,4) rectangle (1,6);
    \node[] at (0,5) {$r_i$};
    
    \draw[gray] (-5,4) rectangle (-3,6);
    \node[] at (-4,5) {$l_i$};
    
    \draw[gray] (-5.5,3.5) rectangle (1.5,7);
    \node[] at (-2,5.75) {$l_j$};
    
    \draw[gray] (-1,-2) rectangle (1,-4);
    
    \node[draw,opacity=0] at (5,0) {.};
\end{tikzpicture}\begin{tikzpicture}[scale=0.250,baseline=-0.2ex]
    \node[] at (0,8.5) {$T'$};
    
    \draw (-4,4) -- (0,0);
    \draw (2,2)--(0,4);
    \draw (4,4)-- (0,0);
    \draw (0,0) -- (0,-2);
    
    \draw[->] (-9,2) -- (-6,2);
    
    \node[right] at (2,1.4) {$\ell_j$};
    \node[right] at (0,-0.8) {$\ell_i$};
    
    \draw[gray] (3,4) rectangle (5,6);
    \node[] at (4,5) {$r_j'$};
    
    \draw[gray] (-1,4) rectangle (1,6);
    \node[] at (0,5) {$l_j'$};
    
    \draw[gray] (-5,4) rectangle (-3,6);
    \node[] at (-4,5) {$l_i'$};
    
    \draw[gray] (-1.5,3.5) rectangle (5.5,7);
    \node[] at (2,5.75) {$r_i'$};
    
    \draw[gray] (-1,-2) rectangle (1,-4);
\end{tikzpicture}
             \caption{Tree rotation associated to the edge connecting $\ell_i$ and $\ell_j$ in $T$.}
             \label{fig:associahedraleaves1}
         \end{figure}

	Now consider the tree $T' \in Y_n$ resulting from performing a tree rotation on the edge connecting $\ell_i$ and $\ell_j$ in $T$ (Figure~\ref{fig:associahedraleaves1}).
	Let $l_k', r_k'$ be the numbers of leaves to the left and right of node $\ell_k$ in $T'$ for $1\le k\le n$.
    If $\ell_k$ is a descendent of $\ell_i$ or a right descendent of $\ell_j$ in $T$, then its position relative to its descendant leaves does not change under the rotation.
    If $\ell_k$ is below $\ell_j$ in $T$, then the sub-tree where rotation occurs is either entirely to its left or entirely to its right.
	Thus for $k\ne i,j$, we have $l_k=l_k'$ and $r_k=r_k'$.
	Moreover $l_i' = l_i$, $r_i' = r_i + r_j$, $l_j' = r_i$, and $r_j' = r_j$ as illustrated in Figure~\ref{fig:associahedraleaves1}.
	From this relation between the leaves of $T$ and $T'$ we can see that
         \begin{align*}
             v_\ell(T) &= (l_i r_i) e_i + (l_j r_j) e_j + \sum_{k \in [n] \setminus \{i,j\}} l_k r_k e_k \\
             &= (l_i r_i) e_i + ((l_i+r_i) r_j) e_j + \sum_{k \in [n] \setminus \{i,j\}} l_k r_k e_k
         \end{align*}
         and 
         \begin{align*}
             v_\ell(T') &= (l_i' r_i') e_i + (l_j' r_j') e_j + \sum_{k \in [n] \setminus \{i,j\}} l_k' r_k' e_k \\
             &= (l_i (r_i+r_j)) e_i + (r_i r_j) e_j + \sum_{k \in [n] \setminus \{i,j\}} l_k r_k e_k.
         \end{align*}
         Subtracting these gives the edge direction oriented from $v_\ell(T)$ to $v_\ell(T')$:
         \begin{align*}
             v(T') - v(T) &= (l_i (r_i+r_j)) e_i + (r_i r_j) e_j - \left( (l_i r_i) e_i + ((l_i+r_i) r_j) e_j \right) \\
             &= (l_i r_i) e_i+ (l_i r_j) e_i + (r_i r_j) e_j - (l_i r_i) e_i- (l_i r_j) e_j - (r_i r_j) e_j \\
             &= (l_i r_j) e_i - (l_i r_j) e_j \\
             &= (l_i r_j)(e_i - e_j)
         \end{align*}
         Thus $i \geq j$ in $\poset_{v_\ell(T)}(a_\ell)$.
        
         \item Now suppose that $\ell_i$ is a right child of $\ell_j$, meaning we have the local sub-tree
         
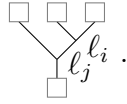
\[
             \begin{tikzpicture}[scale=0.50,baseline=-0.2ex]
    \draw (-1,1) -- (0,0);
    \draw (0.5,0.5)--(0,1);
    \draw (1,1)-- (0,0);
    \draw (0,0) -- (0,-0.5);
    
    \node[right] at (0.5,0.3) {$\ell_i$};
    \node[right] at (0,-0.2) {$\ell_j$};
    
    \draw[gray] (0.75,1) rectangle (1.25,1.5);
    \draw[gray] (-0.25,1) rectangle (0.25,1.5);
    \draw[gray] (-1.25,1) rectangle (-0.75,1.5);
    
    \draw[gray] (-0.25,-0.5) rectangle (0.25,-1);
\end{tikzpicture}.
         \]
         The argument is analogous to the previous case, except we must reverse the roles of $T$ and $T'$ to find that the edge direction is $l_j r_i (e_i - e_j)$.
         This still gives the same covering relation $i \geq j$ in $\poset_{v_\ell(T)}(a_\ell)$.
    \end{enumerate}

    Note that in addition to the edge directions of the associahedron, we have also recovered the lengths of the edges $l_jr_i$.
    Thus the covering relations in $\poset_{v_\ell(T)}(a_\ell)$ are exactly those of $p_\ell(T)\in \operatorname{RBT}_\ell$.
    \end{proof}
        
    
    For a given $p\in \operatorname{RBT_\ell}$, a \Def{maximal element} of $p$ is an element not covered by any other element.
    We use $\mv{p}$ to denote the number of maximal elements in $p$.
    Define the subposet $p_{\geq v}:= \{w \in p : w \geq v \}$ with the order inherited from the poset $p$.
    An element $v$ of $p$ is a \Def{symmetric element} if it is covered by two elements $v_1$ and $v_2$ where the sub-posets $p_{\geq v_1}$ and $p_{\geq v_2}$ are isomorphic.
    We will denote the number of symmetric elements by $\sv{p}$.

    To get an expression for the Brion map on the Hopf algebra of associahedra, when applying the Fock functor, on $\nA$ the associahedra resulting in the same set when forgetting the linear order $\ell$ are identified and when applied on $\P$ the posets $\operatorname{RBT}_\ell$ which are relabelings are identified.
    Hence, up to relabelling, many terms appear more than once in equation~\eqref{eq:brionassoc}.
    Let $\operatorname{RBT_n}$ be the set of all \emph{unlabelled} RBT posets on $n$ elements.
    These are obtained by taking the pruned representations of rooted binary trees with $n+1$ leaves, forgetting the labels, and then forgetting the distinction between left and right children, see Figure~\ref{fig: RBTPoset}.

 \begin{figure}
    \centering
    \scalebox{0.7}{
    \begin{tikzpicture}[nod/.style={draw=black,fill=black,circle,inner sep=1.5pt}, cover/.style={ thick, norange}, scale=0.45]

\def\x{15}
\def\y{9}
\def\z{0}

\draw (8,10) -- (8,-3);

\draw (8+\x,10) -- (8+\x,-3);
\draw (-2+\x,8) -- (12+\x,8);

\draw (-2+\x,4) -- (12+\x,4);

\draw (-2,8) -- (12,8);

\node at (3,9) {$T$};
\node at (10,9) {$p(T)$};

\node at (3+\x,9) {$T$};
\node at (10+\x,9) {$p(T)$};

\draw (2,0) node[nod]{} -- (1,1) node[nod]{} -- (0,2) node[nod]{} -- (-1,3) node[nod] {}; 

\draw (1,4) node[nod]{} -- (2,5) node[nod]{} -- (1,6) node[nod]{} -- (0,7) node[nod] {}; 

\draw (4,4) node[nod]{} -- (5,5) node[nod]{} -- (4,6) node[nod]{} -- (5,7) node[nod] {}; 

\node at (4.5,1.5) {$\cdots$};

\draw (10,2) node[nod]{} -- (10,3) node[nod]{} -- (10,4) node[nod]{} -- (10,5) node[nod]{};

\draw (0+\x,-2+\y) node[nod]{} -- (1+\x,-3+\y) node[nod]{} -- (2+\x,-2+\y) node[nod]{};
\draw (1+\x,-3+\y) -- (2+\x,-4+\y) node[nod]{};

\draw (4+\x,-2+\y) node[nod]{} -- (5+\x,-3+\y) node[nod]{} -- (6+\x,-2+\y) node[nod]{};
\draw (5+\x,-3+\y) -- (4+\x,-4+\y) node[nod]{};

\draw (9+\x,-2+\y+\z) node[nod]{} -- (10+\x,-3+\y+\z) node[nod]{} -- (11+\x,-2+\y+\z) node[nod]{};
\draw (10+\x,-3+\y+\z) -- (10+\x,-4+\y+\z) node[nod]{};

\draw (0+\x,-7+\y) node[nod]{} -- (1+\x,-8+\y) node[nod]{} -- (2+\x,-7+\y) node[nod]{};
\draw (2+\x,-7+\y) -- (1+\x,-6+\y) node[nod]{};

\draw (4+\x,-7+\y) node[nod]{} -- (5+\x,-8+\y) node[nod]{} -- (6+\x,-7+\y) node[nod]{};
\draw (6+\x,-7+\y) -- (7+\x,-6+\y) node[nod]{};

\draw (0+\x,-10+\y) node[nod]{} -- (1+\x,-11+\y) node[nod]{} -- (2+\x,-10+\y) node[nod]{};
\draw (0+\x,-10+\y) -- (-1+\x,-9+\y) node[nod]{};

\draw (4+\x,-10+\y) node[nod]{} -- (5+\x,-11+\y) node[nod]{} -- (6+\x,-10+\y) node[nod]{};
\draw (4+\x,-10+\y) -- (5+\x,-9+\y) node[nod]{};

\draw (9+\x,-8.5+\y) node[nod]{} -- (10+\x,-9.5+\y) node[nod]{} -- (11+\x,-8.5+\y) node[nod]{};
\draw (9+\x,-8.5+\y) -- (9+\x,-7.5+\y) node[nod]{};

\end{tikzpicture}
    }
    \caption{Unlabelled rooted binary trees on the left and posets on the right.}
    
    \label{fig: RBTPoset}
\end{figure}

 \begin{theorem}\label{thrm:brionassociahedraalgebra}
    The Brion map on the Hopf algebra of associahedra is given by
    $$
    B(a_n) = \sum_{p\in\operatorname{RBT_n}} 2^{n-\mv{p}-\sv{p}} \, p \qquad \text{for } n \in \mathbb{N}
    $$
    where $\mv{p}$ and $\sv{p}$ are the numbers of maximal and symmetric elements, respectively, of the poset $p$.
\end{theorem}
\begin{proof}
    Let $p\in\operatorname{RBT}_n$.
    We want to count how many rooted binary trees $T \in Y_n$ turn into the poset diagram for $p$ when we forget the distinction between left and right children.
    Working backwards from an RBT poset $p$, we need to count the number of distinct ways to assign left and right branches to each node of $p$.
    If a node is maximal in $p$, then no assignment of left and right branches can be made.
    Likewise no assignment needs to be made for symmetric nodes, since both possibilities would result in the same RBT.
    All other nodes in $p$ have two ways to assign left and right branches, and these choices are all distinct.
    Thus in total there are $2^{n-\mv{p}-\sv{p}}$ different ways to assign left and right branches to the poset $p$ and obtain an RBT in $Y_n$.
    This gives us the coefficient for any RBT poset which appears in the Brion map, and since all rooted binary tree posets appear, we can conclude that 
         $$B(a_{n}) = \sum_{p\in \operatorname{RBT}_n} 2^{n-\mv{p}-\sv{p}} p.$$
\end{proof}
    
For an example of the previous theorem, consider Figure~\ref{fig:briona3example} which evaluates the Brion map for $a_3$. 

\begin{figure}
    \centering
    \scalebox{1}{
    \begin{tikzpicture}[nod/.style={draw=black,fill=black,circle,inner sep=1.5pt}, cover/.style={ thick, norange}, scale=0.6]

\def\x{9.5}
\def\y{2.5}

\node at (0,0) {$B(a_3) =$};
\node at (2.5,0) {$2^{3-1-0}$};

\draw[thick] (4,-1) node[nod] {} -- (4,0) node[nod] {} -- (4,1) node[nod] {};

\node at (4.5,0) {$+$};
\node at (6,0) {$2^{3-2-1}$};

\draw[thick] (7.5,0.5) node[nod] {} -- (8.5,-0.5) node[nod]{} -- (9.5,0.5) node[nod] {};

\node at (1.25+\x,-2.5+\y){$= 4$};

\draw[thick] (2.5+\x,-3.5+\y) node[nod] {} -- (2.5+\x,-2.5+\y) node[nod] {} -- (2.5+\x,-1.5+\y) node[nod] {};
\node at (3.5+\x,-2.5+\y) {$+$};
\node at (4+\x,-2.5+\y) {$1$};
\draw[thick] (5+\x,-2+\y) node[nod] {} -- (6+\x,-3+\y) node[nod]{} -- (7+\x,-2+\y) node[nod] {};
\end{tikzpicture}
    }
    \caption{The Brion map (for the Hopf algebra) on $a_3$.}
    \label{fig:briona3example}
\end{figure}

We also obtain the following enumerative corollary.
\begin{corollary}\label{cor:catalan}
    The $n$th Catalan number can be expressed as 
    \[
        C_n = \sum_{p\in \operatorname{RBT}_n} 2^{n - \mv{p} - \sv{p}}.
    \]
\end{corollary}

The result above allows for a more visual proof of a classical result.
In essence, you can only have the tree given in Figure~\ref{fig:symmetricRBTposet} when the number of vertices is $n=2^k -1$.

\begin{corollary}
    The $n$th Catalan number $C_n$ is odd if and only if $n=2^k -1$ for some $k \in \mathbb{N}$.
\end{corollary}
\begin{proof}
     We have expressed $C_n$ as a sum of powers of two which depends on RBT posets.
     For a given poset $p$, we have $n-\mv{p} - \sv{p} = 0$ if and only if all nodes are either symmetric or maximal.
     This means every node has either $0$ or $2$ children.
     Call an RBT poset $p$ with this property \Def{totally symmetric}, see Figure~\ref{fig:symmetricRBTposet} for an example.
     By Corollary~\ref{cor:catalan} it is clear that $C_n$ is odd if and only if there are an odd number of totally symmetric posets in $\operatorname{RBT}_n$.

    By definition, all totally symmetric RBT posets must be constructed as follows.
    We start with a root (level 1), which has two children (level 2), which each have two children (level 3), and so on.
    If we repeat this process to level $k$, then the resulting has $2^k-1$ nodes, each of which is either maximal or symmetric.
    
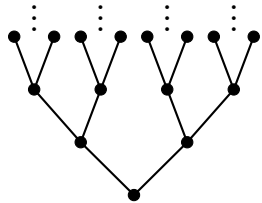
\begin{figure}
        \centering
        \begin{tikzpicture}[vertex/.style={draw=black,fill=black,circle,inner sep=1.5pt}, scale=0.35]

\coordinate (a) at (0,0);

\coordinate (b1) at (2,2);
\coordinate (b2) at (-2,2);

\coordinate (c1) at (3.75,4);
\coordinate (c2) at (1.25,4);
\coordinate (c3) at (-1.25,4);
\coordinate (c4) at (-3.75,4);

\coordinate (d1) at (-4.5,6);
\coordinate (d2) at (-3,6);
\coordinate (d3) at (-2,6);
\coordinate (d4) at (-0.5,6);

\coordinate (d5) at (0.5,6);
\coordinate (d6) at (2,6);
\coordinate (d7) at (3,6);
\coordinate (d8) at (4.5,6);

\node[vertex] at (a) {};
\node[vertex] at (b1) {};
\node[vertex] at (b2) {};
\node[vertex] at (c1) {};
\node[vertex] at (c2) {};
\node[vertex] at (c3) {};
\node[vertex] at (c4) {};
\node[vertex] at (d1) {};
\node[vertex] at (d2) {};
\node[vertex] at (d3) {};
\node[vertex] at (d4) {};
\node[vertex] at (d5) {};
\node[vertex] at (d6) {};
\node[vertex] at (d7) {};
\node[vertex] at (d8) {};

\draw[thick] (a) -- (b1);
\draw[thick] (a) -- (b2);

\draw[thick] (b1) -- (c1);
\draw[thick] (b1) -- (c2);

\draw[thick] (b2) -- (c3);
\draw[thick] (b2) -- (c4);

\draw[thick] (c1) -- (d8);
\draw[thick] (c1) -- (d7);
\draw[thick] (c2) -- (d6);
\draw[thick] (c2) -- (d5);
\draw[thick] (c3) -- (d4);
\draw[thick] (c3) -- (d3);
\draw[thick] (c4) -- (d2);
\draw[thick] (c4) -- (d1);

\node at (1.25,7) {\Large$\mathbf{\vdots}$};
\node at (-1.25,7) {\Large$\mathbf{\vdots}$};
\node at (3.75,7) {\Large$\mathbf{\vdots}$};
\node at (-3.75,7) {\Large$\mathbf{\vdots}$};
\end{tikzpicture}
        \label{fig:symmetricRBTposet}
        \caption{A totally symmetric RBT poset.}
    \end{figure}
    Thus the number of totally symmetric posets in $\operatorname{RBT}_n$ is $1$ if $n=2^k-1$ for some $k\in \mathbb{N}$ and $0$ otherwise.
    \end{proof}

\subsection{The Brion Map on Orbit Polytopes}\label{subsec:brionmaporbit}

We now consider the Brion map on orbit polytopes, making use of the correspondence between orbit polytopes and integer compositions from Section~\ref{sec:orbitpolytopes}.
Given a composition $\lambda = (\lambda_1, \lambda_2, \ldots,\lambda_k)$ of the integer $n$, note that the number of vertices of an orbit polytope in the class $\mathcal{O}_\lambda$ is $\frac{n!}{\lambda_1! \cdots \lambda_k!}$.
Construct the unlabelled ${\lambda}$\Def{-layered poset} $p(\lambda)$ with $\lambda_1$ nodes on the first level, $\lambda_2$ on the level above that, and so on until the top level which has $\lambda_k$ nodes; include all possible covering relations between each consecutive pair of levels, see Figure~\ref{fig:lambdalayerposetexample} for an example.

\begin{figure}
    \centering
    \scalebox{0.6}{
    \begin{tikzpicture}[
	edge/.style={color=ngreen},
	facet/.style={fill=ngreen,fill opacity=0.300000}],scale=0.75]


\foreach \x in {0,...,2}{ 
    \coordinate[draw=black,fill=black,circle,inner sep=1.5pt] (\x0) at (\x-0.5*2,0);
    }

\foreach \x in {0,...,0}{
    \coordinate[draw=black,fill=black,circle,inner sep=1.5pt] (\x1) at (\x-0.5*0,1);
    }

\foreach \x in {0,...,4}{
    \coordinate[draw=black,fill=black,circle,inner sep=1.5pt] (\x2) at (\x-0.5*4,2);
    }
    
\foreach \x in {0,...,1}{
    \coordinate[draw=black,fill=black,circle,inner sep=1.5pt] (\x3) at (\x-0.5*1,3);
    }

\foreach \x in {0,...,4}{
    \coordinate[draw=black,fill=black,circle,inner sep=1.5pt] (\x4) at (\x-0.5*4,4);
    }


\foreach \y in {0,...,2}{
    \draw (\y0) -- (01); 
}

\foreach \x in {0,...,4}{
    \foreach \y in {0,...,0}{
        \draw (\y1) -- (\x2); 
    }
}

\foreach \x in {0,...,1}{
    \foreach \y in {0,...,4}{
        \draw (\y2) -- (\x3);
    }
}

\foreach \x in {0,...,4}{
    \foreach \y in {0,...,1}{
        \draw (\y3) -- (\x4);
    }
}

\end{tikzpicture}
    }
    \caption{The $\lambda$-layered poset for $\lambda = (3,1,5,2,5)$.}
    \label{fig:lambdalayerposetexample}
\end{figure}

\begin{theorem}\label{thm:brionmaporbit}
    For $\Or_\lambda\in\nOP[I]$, the Brion map on the Hopf monoid of orbit polytopes is given by
    \[
        B( \mathcal{O}_\lambda ) = \sum_{\substack{p\text{ a labelling}\\\text{of }p(\lambda)\text{ by }I}} p.
    \]
\end{theorem}
\begin{proof}
    We focus on the case where $I=\{1,2,\dots,n\}$; the proof for general sets $I$ is analogous.
    Recall that $\mathcal{O}_\lambda$ is a normal equivalence class of orbit polytopes, as defined in Section~\ref{subsec:gp}.
     Normally equivalent polytopes have the same vertex cones, so consider the orbit polytope $O\subseteq \mathbb{R}^{[n]}$ of the normal equivalence class $\Or_\lambda$ whose vertices are all permutations of the vertex 
     \[
     v = (\underbrace{k,\ldots,k}_{\lambda_1}, \ldots, \underbrace{2,\ldots,2}_{\lambda_{k-1}}, \underbrace{1,\ldots,1}_{\lambda_k}).
     \]
     We claim that $\poset_v(O) = p(\lambda)$. A neighbor $w$ of $v$ is obtained by swapping the positions of two consecutive values in $v$; i.e. by swapping one of the $1$s with one of the $2$s, one of the $j$'s with one of the $j+1$'s, etc.
     Suppose $w$ is obtained by swapping a $k-(i-1)$ with a $k-i$ in $v$ for some $1\le i< k$:
     $$
      w=(\ldots,\underbrace{k-(i-1),\ldots,\textcolor{blue}{k-i},\ldots,k-(i-1)}_{ \lambda_{i}},\underbrace{k-i,\ldots,\textcolor{blue}{k-(i-1)},\ldots,k-i}_{\lambda_{i+1}},\ldots)
     $$
     This means that the positions where $v$ and $w$ differ are $\lambda_1 + \cdots + \lambda_{i-1}+a$ and $\lambda_1 + \cdots + \lambda_{i}+b$ for some $1\le a \le \lambda_i$ and $1 \leq b \leq \lambda_{i+1}$, and the edge direction vector $w-v$ is $e_{\lambda_1 + \cdots + \lambda_{i}+b} - e_{\lambda_1 + \cdots + \lambda_{i-1}+a}$.
     Hence we have the relation $\lambda_1 + \cdots + \lambda_{i}+b \geq \lambda_1 + \cdots + \lambda_{i-1}+a$ in $\poset_v(O)$ for any choices of $1 \leq a \leq \lambda_i$ and $1 \leq b \leq \lambda_{i+1}$.
     From this we can deduce that $\poset_v(O)$ is the $\lambda$-layered poset labelled with $1,2,\dots,\lambda_1$ on the first level, $\lambda_1+1,\lambda_1+d,\dots,\lambda_1+\lambda_2$ on the second level, and so on.
     
     To find the vertex cone for another vertex of $O$, we can permute the coordinates of points in $O$ by any element of $S_n$.
     This would permute the labels of corresponding poset, but would maintain the underlying unlabelled structure.
     Thus the vertex posets are all labellings of $p(\lambda)$.
\end{proof} 

\begin{theorem}
    The Brion map on the Hopf algebra of orbit polytopes is given by,
    $$B( \mathcal{O}_\lambda ) = \frac{n!}{\lambda_1! \cdots \lambda_k!} p(\lambda).$$
\end{theorem}
\begin{proof}
    Choose $\lambda_1$ labels for the bottom level of $p(\lambda)$, $\lambda_2$ labellings for level 2, and so on.
    Thus the number of such labellings is the multinomial coefficient $\frac{n!}{\lambda_1! \cdots \lambda_k!}$.
    Alternatively, this is the number of vertices of a representative of $\Or_\lambda$.
\end{proof}
\section{The Dual Brion Map for \texorpdfstring{$\nPI$, $\nA$, and $\nOP$}{permutahedra, associahedra, and orbit polytopes}}\label{sec:DualBrionMap}

The \Def{dual Brion map} $B^*:\mathbf{P}^*\to\GP^*$ is the map dual to the Brion map $B:\GP \to \P$ which is defined linearly on the basis elements by $p^* \mapsto p^* \circ B$.
This is again a morphism of Hopf monoids.
 
Since $\mathbf{P}^*$ and $\GP^*$ are cocommutative, by Theorem~\ref{thrm:CMM} we may view these Hopf monoids as the universal enveloping monoids of their respective primitives.
Hence we do not lose information by considering the restriction of the dual Brion map to the primitive elements:
$B^*\vert_{\mathcal{P}(\mathbf{P}*)}:\mathcal{P}(\mathbf{P}^*)\to\mathcal{P}(\GP^*).$

We can further restrict the image of this dual map to specific submonoids $\nPI^*, \nA^*, \nOP^*$ which is a dual notion to what we did in Section~\ref{sec:brionmap} when restricting our domain. Because we have inclusion maps $\iota_{\nPI}: {\nPI} \to {\nGP}$, $\iota_{\nA}: {\nA} \to {\nGP}$, $\iota_{\nOP}: {\nOP} \to {\nGP}$
which are injective, the dual maps $\iota_{\nPI}^*,\, \iota_{\nA}^*,\, \iota_{\nOP}^*$ are surjective. 
We can again restrict $\iota^*$ to the primitives. Denote the following compositions 
\[
B_{\nPI}^*=\iota_{\nPI}^*\vert_{\mathcal{P}(\nGP^*)} \circ B^*\vert_{\mathcal{P}(\mathbf{P}^*)},
\quad B_{\nA}^*=\iota_{\nA}^*\vert_{\mathcal{P}(\nGP^*)} \circ B^*\vert_{\mathcal{P}(\mathbf{P}^*)}, \quad B_{\nOP}^*=\iota_{\nOP}^*\vert_{\mathcal{P}(\nGP^*)} \circ B^*\vert_{\mathcal{P}(\mathbf{P}^*)}.
\]
This dual inclusion map selects the terms in the image of $B^*$ which are the linear functionals associated to the primitives in each respective Hopf monoid.
Finally, composing with the Fock functor allows us to pass to Hopf algebras and define the maps we will now explore: $B_{\Pi}^* = \mathcal{F}( B_{\nPI}^*), B_{A}^* = \mathcal{F} (B_{\nA}^*), B_{OP}^*=\mathcal{F} (B_{\nOP}^*)$.  

Since the primitives given by elements $p^*\in\mathbf{P}^*$ are dual to connected posets $p\in\mathbf{P}$, we have that $B^*(p^*)$ is a linear functional sending $Q\in\nGP$ to the coefficient of the poset $p$ in the evaluation of the Brion map $B(Q)$.
If $Q\in \pp{\nGP}$ is a single generalized permutahedron, then $B^*(p^*)(Q)$ will be the coefficient of the Brion map if the vertex poset of $Q$ at some vertex is the same as $p$, and $0$ otherwise.

\begin{theorem}\label{thm:dualbrionmap}
    The dual Brion map restricted to permutahedra, associahedra, and orbit polytopes is given as follows:
    \begin{enumerate}[label=(\roman*)]
        \item The dual Brion map over the dual Hopf algebra of permutahedra is defined as
    \[B_{\Pi}^*(p^*) = \begin{cases}
        n! \, \pi_n^* & \text{if } p \text{ is an unlabeled chain of } n \text{ nodes}, \\
        0 & \text{otherwise}.
    \end{cases}\]

    \item The dual Brion map over the dual Hopf algebra of associahedra is defined as
    \[
    B_{A}^*(p^*) = \begin{cases}
        2^{n-\mv{p}-\sv{p}} \, a_n^* & \text{if } p \text{ is an unlabeled RBT poset on } [n], \\
        0 & \text{otherwise.}
    \end{cases} \]

    \item The dual Brion map over the dual Hopf algebra of orbit polytopes is defined as
    \[
    B_{OP}^*(p^*) = \begin{cases}
        \frac{n!}{\lambda_1! \cdots \lambda_k!} \, \Or_\lambda^* & \text{if } p \text{ is } P(\lambda) \text{ for some } \lambda, \\
        0 & \text{otherwise.}
    \end{cases}
    \]
    \end{enumerate}
\end{theorem}

\begin{proof}
    When considering the Brion map over Hopf algebras, for a given permutahedron $\pi_n$, associahedron $a_n$, and orbit polytope $\Or_\lambda$ we have
    \[
    B(\pi_n) = n!\, c_n, \quad B(a_n) = \sum_{\text{RBT poset } p} 2^{n-\mv{p}-\sv{p}} \, p, \quad B(\Or_\lambda) = \frac{n!}{\lambda_1! \cdots \lambda_k!} p(\lambda)
    \]
    where $c_n$ is the unlabeled chain on $n$ nodes and $\lambda = (\lambda_1,\ldots,\lambda_k)$ (see Section~\ref{sec:brionmap}).
    We now consider $B^*(p^*)$ which in each case is a linear functional from the Hopf algebra ($\Pi$, $A$, or $OP$) to $\mathbb{R}$.
    Since the primitives in $\GP^*$ corresponded to those generalized permutahedra in $\GP$ which are not products by Theorem~\ref{thm:GPprimitives} we only consider the linear functional evaluated on these elements of $\Pi$, $A$, and $OP$.
    If $p$ is a poset appearing in the Brion map...
    \begin{enumerate}[label=(\roman*)]
        \item ...of $\pi_n\in\Pi$, then $p$ is the unlabeled chain $c_n$ on $n$ nodes, and hence
        \[
            \left(B_{\Pi}^*(p^*)\right)(\pi_n) = p^* (B(\pi_n)) = p^*(n! \, p) = n!
        \]
         Moreover, $\left(B^*(p^*)\right)(\pi_m) = 0$ for $m \neq n$ since $p$ has $n$ nodes but the posets which appear in $B(\pi_m)$ have $m$ nodes.

         \item ...of $a_n\in A$, then $p$ is an unlabeled RBT poset on $n$ nodes, and thus
         \[
         \left(B_{A}^*(p^*)\right)(a_n)= p^*\left( B(a_n) \right) = \sum_{\text{RBT poset } q} 2^{n-\mv{q}-\sv{q}} \, p^*(q) = 2^{n-\mv{p}-\sv{p}} \]
         Moreover, $\left(B^*(p^*)\right)(a_m) = 0$ for $m\ne n$ by the same reasoning as above.

         \item ...of $\Or_\lambda\in OP$, then $p$ is the unlabeled $\lambda$-layered poset $p(\lambda)$ and
         \[
            \left(B_{OP}^*(p^*)\right)(\Or_\lambda)= p^*\left( B(\Or_\lambda) \right) = p^*\left( \frac{n!}{\lambda_1! \cdots \lambda_k!}\, p\right) = \frac{n!}{\lambda_1! \cdots \lambda_k!}
         \]
         Moreover if $\lambda'$ is a different composition than $\lambda$, we have $\left(B^*(p^*)\right)(\Or_{\lambda'}) = 0$.
    \end{enumerate}
    If $p$ is any other poset not appearing as a term in the evaluation of the Brion map, then the linear functional $B^*(p^*)$ is identically $0$. 

\end{proof}

\section{Acknowledgements}
We thank Federico Ardila for valuable guidance and mentorship, and Tracy Camacho for fruitful conversations.
This work first appeared in the first author's Masters thesis, which was co-advised by Federico Ardila and the second author at San Francisco State University.
The second author was partially supported by grant 2018-03968 of the Swedish Research Council as well as the G\"oran Gustafsson Foundation.

\bibliographystyle{plain}
\bibliography{references}

\end{document}